\documentclass[11pt]{amsart}
\usepackage{amssymb,amsthm,amsmath}
\RequirePackage[dvipsnames,usenames]{xcolor}
\usepackage{hyperref}
\usepackage{mathtools,etoolbox}
\usepackage[all]{xy}
\usepackage{tikz}
\usepackage{enumitem}
\usepackage[english]{babel}
\usepackage{graphics}
\usepackage{color}

\hypersetup{
bookmarksdepth=3,
bookmarksopen,
bookmarksnumbered,
pdfstartview=FitH,
colorlinks,
linkcolor=Sepia,
anchorcolor=BurntOrange,
citecolor=MidnightBlue,
citecolor=OliveGreen,
filecolor=BlueViolet,
menucolor=Yellow,
urlcolor=OliveGreen
}

\newcommand{\factor}[2]{\left. \raise 2pt\hbox{\ensuremath{#1}} \right/
        \hskip -2pt\raise -2pt\hbox{\ensuremath{#2}}}

\numberwithin{equation}{section}

\makeatletter
\renewcommand\subsection{
  \renewcommand{\sfdefault}{pag}
  \@startsection{subsection}%
  {2}{0pt}{.8\baselineskip}{.4\baselineskip}{\raggedright
    \sffamily\itshape\small\bfseries
  }}
\renewcommand\section{
  \renewcommand{\sfdefault}{phv}
  \@startsection{section} %
  {1}{0pt}{\baselineskip}{.8\baselineskip}{\centering
    \sffamily
    \scshape
    \bfseries
}}
\makeatother

\usepackage{ifthen}
\usepackage{longtable}

\usepackage[left=1.02in,top=1.0in,right=1.02in,bottom=1.0in]{geometry}
\usepackage{multirow}

\setlist[enumerate]{leftmargin=0.8cm}
\setlist[itemize]{leftmargin=0.8cm}
\setlist[description]{leftmargin=0.0cm}

\theoremstyle{plain}
\newtheorem{theorem}{Theorem}[section]

\newtheorem{proposition}[theorem]{Proposition}

\newtheorem{corollary}[theorem]{Corollary}
\newtheorem{def-thm}[theorem]{Definition-Theorem}
\newtheorem{lemma}[theorem]{Lemma}

\newtheorem{defi}[theorem]{Definition}
\newtheorem{definition}[theorem]{Definition}

\theoremstyle{definition}

\newtheorem{remark}[theorem]{Remark}

\usepackage{color}

\DeclareMathOperator{\codim}{codim}
\DeclareMathOperator{\supp}{Supp\hspace{1pt}}

\def\min{\mathop{\mathrm{min}}}
\def\CC{\mathbb C}

\def\Q{\mathbb Q}

\def\PP{\mathbb P}

\let\e\epsilon

\let\l\lambda

\newcommand{\R}{{\mathbb R}}

\newcommand{\calC}{{\mathcal C}}

\newcommand{\calL}{{\mathcal L}}

\newcommand{\calO}{{\mathcal O}}

\begin{document}
\title{{Divisibility of polynomials and degeneracy of integral points}}
\author{Erwan Rousseau}
\address{ Institut Universitaire de France \& Aix Marseille Univ\newline \indent
	CNRS, Centrale Marseille, I2M, Marseille, France }
\email{erwan.rousseau@univ-amu.fr}
\author{Amos Turchet}
\address{Dipartimento di Matematica e Fisica, Università degli studi Roma 3,\newline  \indent L.go S. L. Murialdo 1, 00146 Roma, Italy }
\email{amos.turchet@uniroma3.it}
\author{Julie Tzu-Yueh Wang}
\address{Institute of Mathematics, Academia Sinica \newline
\indent No.\ 1, Sec.\ 4, Roosevelt Road\newline
\indent Taipei 10617, Taiwan}
\email{jwang@math.sinica.edu.tw}
\begin{abstract}
  We prove several statements about arithmetic hyperbolicity of certain blow-up varieties. As a corollary we obtain multiple examples of simply connected quasi-projective varieties that are pseudo-arithmetically hyperbolic. This generalizes results of Corvaja and Zannier obtained in dimension 2 to arbitrary dimension. The key input is an application of the Ru-Vojta's strategy. We also obtain the analogue results for function fields and Nevanlinna theory with the goal to apply them in a future paper in the context of Campana's conjectures.
  \end{abstract}
 \thanks{2010\ {\it Mathematics Subject Classification.}
11J87, 11J97, 14G05, 32A22.}

\maketitle 


\section{introduction}\label{introduction}

The goal of this project is to generalize the results of our previous paper \cite{RTW} to higher dimensions. In \cite{RTW} we dealt with two competing conjectures that aim to characterize algebraic varieties defined over a number field $k$ that have a potentially dense set of $k$-rational points. On one hand Campana conjectured that the class of these varieties is the class of \emph{special varieties}, introduced in \cite{Ca04}, while the Weak Specialness Conjecture (see \cite[Conjecture 1.2]{HT}) predicts that these should be the weakly special varieties, i.e. varieties that do not admit any \'etale cover that dominates a variety of general type. In \cite[Theorem 4.2]{RTW} we constructed examples of quasi-projective threefolds that are not special but weakly-special (see also \cite{BT,CP} for other constructions), and proved in \cite[Theorem 6.5]{RTW} that such examples possess properties that contradict function field and analytic analogues of the Weak-Specialness conjecture.

In order to generalize these results in higher dimensions we need two ingredients: the first one, that is the focus of the present paper, is the construction of simply connected varieties $X$ where we have a good control on the distribution of integral points and entire curves. The second one, which will be addressed in a forthcoming paper, is the construction of weakly-special varieties $Z$ fibered over $X$ and the study of the orbifold hyperbolicity of the base $X$.

In \cite{RTW} we used as ``arithmetic input'' a construction of Corvaja and Zannier in \cite{Corvaja:Zannier:2010} of a simply connected quasi-projective surface whose integral points are not Zariski dense. The key observation in \cite{Corvaja:Zannier:2010} was that the study of the distribution of integral points in such surfaces is connected to divisibility problems of polynomials evaluated at $S$-integers. In fact many classical problems in Diophantine Geometry, such as Siegel's finiteness theorem or the $S$-unit equation, can be rephrased via divisibility of polynomials. In this paper we use this observation to obtain several new results that extend \cite{Corvaja:Zannier:2010} to higher dimensions.

The first result is a generalization of \cite[Theorem 4]{Corvaja:Zannier:2010} to an arbitrary number of variables.

\begin{theorem}\label{keyprojective}
    Let $n\ge 2$.
    Let $k$ be a number field, let $S$ be a finite set of places including the Archimedean ones, and  let $\mathcal O_S$ be the ring of $S$-integers.
    Let $F_1,\hdots,F_r, G\in \mathcal O_S[x_0,\hdots,x_n]$ be absolutely irreducible homogeneous polynomials of the same degree.
    Suppose that the hypersurfaces defined by   $F_1,\hdots,F_r$ and $G$ are in general position, i.e. any intersection of  $n+1$   hypersurfaces is empty,  and $\deg (F_i)\ge \deg (G)$ for $ i=1,\hdots,r$.
    Then there exists a closed subset $Z \subset \PP^n$, independent of $k$ and $S$, such that there are only finitely many points $(x_0,\hdots,x_n)\in \PP^n(\calO_S) \setminus Z$ such that one of the following holds:
    \begin{itemize}
    \item[{\rm (i)}]
    $r\ge 2n+1$ and $F_i(x_0,\hdots,x_n)\mid G(x_0,\hdots,x_n)$ in the ring $\mathcal O_S$, for $ i=1,\hdots,r$; or
    \item[{\rm (ii)}]
    $r\ge n+2$ and $\prod_{i=1}^rF_i(x_0,\hdots,x_n) \mid G(x_0,\hdots,x_n)$  in the ring $\mathcal O_S$.
    \end{itemize}
    \end{theorem}

    In \cite[Theorem 4]{Corvaja:Zannier:2010} the original Theorem was obtained in the case $n=2$. Moreover, in Theorem \ref{keyprojective}, we obtain a stronger conclusion, namely the existence of an exceptional set $Z$ independent of the field of definition.
    The above Theorem yields the following Corollary that generalizes the classical $S$-unit equation (that is the case $g = 1$).
    
    \begin{corollary}[Compare with   {\cite[Corollary 1]{Corvaja:Zannier:2010}}]\label{cor:1}
    Let $g\in\mathcal O_S[x_1,\hdots,x_n]$ be a polynomial of degree $\le 1$ such that $g(0,\hdots,0)\ne 0$, $g(1,0,\hdots,0)\ne 0, \hdots,\, g(0,\hdots,0,1)\ne 0$.  The $n$-tuples $(x_1,\hdots,x_n)\in\mathcal O_S^{n}$ such that $\left(\left(1-\sum_{i=1}^n x_i\right)\prod_{i=1}^n x_i\right) \mid g(x_1,\hdots,x_n)$ are not Zariski-dense in $\mathbb A^n$.
    \end{corollary}
    \begin{proof}
    Apply Theorem \ref{keyprojective} (ii) to the linear forms $X_0$,\dots,$X_n$, and $X_0-\sum_{i=1}^n X_i$.
    \end{proof}

    As we will see, both results follow from a more general statement, Theorem \ref{key} in Section \ref{sec:proofofKey}. 

    We mentioned above that divisibility results as the ones of Theorem \ref{keyprojective} and Corollary \ref{cor:1}, are related to degeneracy of integral points on varieties. The first statement in this direction is the following theorem that studies certain blow up of $\PP^n$ along intersections of hypersurfaces.
  
 \begin{theorem}\label{IP'}
Let $n\ge 2$, $r\ge 2n+1$ and  $D_0,D_1, \dots, D_{r} $ be  hypersurfaces  in general position on $\mathbb P^n$ defined over $k$.     
Let $\pi: X\to \mathbb P^n$ be the blowup of the union of subschemes $D_i\cap D_0$, $1\le i\le r$, and let $\widetilde D_i$ be the strict transform of  $D_i$.  Let $D=\widetilde D_1+\cdots+\widetilde D_r$.  Then $ X\setminus D$ is arithmetically pseudo-hyperbolic.
\end{theorem}
 
This is the key result needed for the future applications to the study of weakly special varieties. In fact we can use Theorem \ref{IP'} to construct simply connected varieties whose integral points are not Zariski dense, thus generalizing Corvaja and Zannier's construction in arbitrary dimension.

\begin{proposition}[Compare with {\cite[Theorem 3]{Corvaja:Zannier:2010}}]\label{prop:simply_connected}
In the setting of Theorem \ref{IP'}, suppose that the divisor $D_0+D_1+\dots+D_r$ has simple normal crossing singularities.
Then the variety $X\setminus D$ appearing in Theorem \ref{IP'} is simply connected.
\end{proposition}

\begin{proof}
If $n = 2$ this was done in \cite[Example 4.4]{RTW}. If $n \geq 3$, consider a loop around $\widetilde{D}_i$. Now observe that, if $E$ is the exceptional divisor over $D_i\cap D_0$, the generic fiber of the restriction 
$\pi: E \setminus D \to D_i\cap D_0$  is isomorphic to $\CC$. Thus the loop becomes homotopically trivial in $X\setminus D$.
\end{proof}

Proposition \ref{prop:simply_connected} will be used in a subsequent paper to discuss analogues of a question of Hassett and Tschinkel in \cite[Problem 3.7]{HassettT} for function fields and entire curves.

Along the same lines we generalize to arbitrary dimensions \cite[Corollary 2]{Corvaja:Zannier:2010}.

\begin{theorem}\label{proposition7}
Let $n\ge 2$ and let $H_1,\hdots,H_{2n}$ be $2n$ hyperplanes in general position on $\mathbb P^n$ defined over $k$.   Choose $n+1$ points $P_i$, $1\le i\le n+1$ such that $P_i\in H_i$, $1\le i\le n+1$, and $P_i\notin H_j$ if $i\ne j$ for $1\le j\le 2n$.  Let $\pi: X\to\mathbb P^n$ be the blowup of the $n+1$ points $P_i$, $1\le i\le n+1$, and let $D\subset   X$ be the strict transform of $H_1+\cdots+H_{2n}$.   Then  
$  X\setminus D$ is arithmetically pseudo-hyperbolic.
\end{theorem}

Finally we obtain a generalization of \cite[Proposition 1]{Corvaja:Zannier:2010} and \cite[Theorem 7]{Corvaja:Zannier:2010} as follows.

 \begin{theorem}\label{proposition1}
Let  $n\ge 2$ and $q\ge 3n$ be two integers; for every index $i\in \mathbb Z/q\mathbb Z$, let $H_i$ be a hyperplane in $\mathbb P^n$ defined over $k$.  Suppose that the $H_i$'s are in general position. For each index $i\in \mathbb Z/q\mathbb Z$ let $P_i$ be the intersection point $\cap_{j=0}^{n-1} H_{i+j}$.  Let $\pi: X\to \mathbb P^n$ be the blow-up of the points $P_1,\hdots, P_q$, let $\widetilde H_i\subset  X$ be the strict transform of $H_i$, and let $D=\widetilde H_1+\cdots+\widetilde H_q$.  Then $ X\setminus D$ is arithmetically pseudo-hyperbolic.
\end{theorem}
 
  \begin{theorem}\label{Theorem7}
Let  $n\ge 2$, $q\ge 3n$ be an integer; for every index $i\in \mathbb Z/q\mathbb Z$, let $F_i$, $1\le i\le q$ be linear form in   $k[x_0,\hdots,x_n]$ in general position.  
Then there exists a Zariski closed subset $Z$ of $\mathbb P^n$ such that the set of points $[x_0:\cdots:x_n]\in\mathbb P^n(k)$ satisfying, for each $i\in \mathbb Z/q\mathbb Z$, the equality of ideals
\[
F_i(x_0,\hdots,x_n)\cdot (x_0,\hdots,x_n)=\prod_{j=i-n+1}^i (F_j(x_0,\hdots,x_n),\hdots,F_{j+n-1}(x_0,\hdots,x_n))
\]
is contained in $Z$.
\end{theorem}

We also mention that most of these results can be rephrased as hyperbolicity of complements of divisors in certain varieties that are higher dimensional analogues of Del Pezzo surfaces. For example Theorem \ref{IP'} applies to open subsets of the blow up of $\PP^3$ in $r \geq 7$ lines. Interestingly enough the condition $r \geq 7$ characterizes precisely the blow-ups that are \emph{not} weakly Fano (and hence not Mori dream spaces).
\medskip

\paragraph{\texttt{Ideas of the proof}} The main technical tool to obtain the proof of the previous results, as in our previous paper \cite{RTW}, is to apply (a generalization of) the main theorem of Ru-Vojta (see Theorem \ref{Ru-Vojta}). In fact in \cite{RTW} we have already proven that the Ru-Vojta method can be used to recover the main theorem of \cite{CZAnnals}, that was used in \cite{Corvaja:Zannier:2010} to obtain the degeneracy results that we are generalizing in this paper. However, in this situation, the computations of the constant $\beta$, which is the crucial part of the proof, is less direct and make use of several ingredients, among them an adaptation of Autissier's ideas of \cite{aut}.

Moreover, by carefully tracing the exceptional set, and following a strategy already discussed by Levin in \cite{levin_annal}, we can obtain a stronger result, namely pseudo-arithmetic hyperbolicity instead of degeneracy of integral points. In particular, this shows that in our statements, the closed subset outside of which the integral points are finite, does not depend on the field of definition (as expected in the stronger versions of the conjectures of Lang and Vojta); we refer to \cite[Section 3]{RTW} for more details and discussions. In fact our results are indeed instances of the Lang-Vojta conjecture for integral points.

The paper is organized as follows: in Section \ref{Preliminary} we collect some preliminary definitions and properties of local heights and we link divisibility problems with integral points. In Section \ref{sec:RuVojta} we state the Main Theorem of Ru-Vojta with better control of the exceptional set. In Section \ref{sec:proofofKey} we prove Theorems \ref{keyprojective} and Theorem \ref{IP'}. In Section \ref{sec:prop7} we compute $\beta$ in several cases and we prove Theorem \ref{proposition7}. In Section \ref{sec:proof_last} we prove Theorem \ref{proposition1} and Theorem \ref{Theorem7}. In Section \ref{sec:analytic} we collect the analogue results for holomorphic maps, while in Section \ref{sec:ff} we present the results over function fields, together with the proof of the key Ru-Vojta statement.

\medskip

\paragraph{\texttt{Divisibility and Integral Points}}
As observed in \cite{Corvaja:Zannier:2010}, there is a connection between distribution of integral points on certain rational quasi-projective varieties, and divisibility problems. In fact, Corvaja and Zannier show the following: given points $P_1,\dots,P_n \in \PP^2$ that are the intersection of two curves $\calC_1,\calC_2$, let $\pi: X \to \PP^2$ be the blow up along $P_1,\dots,P_n$. Then, for a point $Q \in \PP^2$, $Q \neq P_i$, one can relate the condition that $\pi^{-1}(Q)$ is integral on $X$ with respect to the strict transform of $\calC_1$, with a divisibility condition for the polynomials defining $C_1$ and $C_2$ (locally). 

We formalize this in Lemma \ref{DivisibilityandIP}, where we generalize to arbitrary dimensions \cite[Lemma 1]{Corvaja:Zannier:2010}.

In fact, divisibility conditions are connected to the celebrated Vojta's conjectures (as in \cite[Conjecture 3.4.3]{vojta_lect}) in many ways: Silverman in \cite{Silverman2005} observed that GCD results for $S$-units in number fields, as in the seminal paper \cite{Bugeaud:Corvaja:Zannier:2003}, are related to Vojta's conjecture on certain blow-ups. Since then, a number of articles have been devoted to exploit this connection. We cite for example \cite{BCT,CaTur,CZGm,GSW,GuoWang,levin_gcd,levinwang,PWgcd,Tur,WaYa,yasumona,yasuams}.

\subsection*{Acknowledgements}
We thank Pietro Corvaja and Umberto Zannier for several discussions. ER was supported by Institut Universitaire de France and the ANR project \lq\lq FOLIAGE\rq\rq{}, ANR-16-CE40-0008. AT was partially supported by PRIN ``Advances in Moduli Theory and Birational Classification'' and is a member of GNSAGA-INdAM. JW was supported in part by Taiwan's MoST grant 108-2115-M-001-001-MY2.

 \section{Heights and integral points}\label{Preliminary}
 We collect here standard facts and definitions on local and global Weil heights and integral points. We refer to  \cite[Chapter 10]{Lang}, \cite[B.8]{HinSil},  \cite[Section 2.3]{levin_gcd} or \cite[Section 2]{Sil} for more details about this section. We have decided to avoid the use of integral models to discuss integral points since it is more natural in the arithmetic context of the Ru-Vojta method.

Let $k$ be a number field and $M_{k}$ be the set of places, normalized so that it satisfies the product formula
\[
\prod_{v\in M_{k}}  |x|_v=1,\qquad \text{ for }  x\in k^\times \text{.}
\]
For a point $[x_0:\cdots:x_n]\in \mathbb P^n(k)$,
the standard logarithmic height is defined by
\[
h([x_0:\cdots:x_n])=\sum_{v\in M_{k}}\log \max\{|x_0|_v,\hdots,|x_n|_v\},
\]
and it is independent of the choice of coordinates $x_0,\hdots,x_n$ by the product formula.

A \emph{$M_k$-constant} is a family $\{\gamma_v\}_{v\in M_k}$ of real numbers, with all but finitely many equal to zero. Equivalently it is a real-valued function $\gamma: M_k \to \R$ which is zero almost everywhere.  Given two families $\{\lambda_{1v}\}$ and $\{\lambda_{2v}\}$, we say $\lambda_{1v} \le \lambda_{2v}$ holds up to an $M_k$-constant if there exists an $M_k$-constant $\{\gamma_v\}$ such that $\lambda_{2v} - \lambda_{1v} \geq \gamma_v$ for all $v \in M_k$.  We say $\lambda_{1v} = \lambda_{2v}$ up to an $M_k$-constant if $\lambda_{1v} \le \lambda_{2v}$ and $\lambda_{2v} \le \lambda_{1v}$ up to $M_k$-constants.

Let $V$ be projective variety defined over a number field $k$.  The classical theory of heights associates to every Cartier divisor $D$ on $V$ a \emph{height function} $h_D:V(k)\to \mathbb R$ and a \emph{local Weil function}  (or \emph{local height function})
$\lambda_{D,v}: V(k)\setminus \supp (D)\to \mathbb R$
for each $v\in M_k$, such that
\[
\sum_{v\in M_{k}}\lambda_{D,v}(P)=h_D(P)+O(1)
\]
for all $P\in V(k)\setminus \supp (D)$.

We also recall some basic properties of local Weil functions associated to closed subschemes from \cite[Section 2]{Sil}.
Given a closed subscheme $Y$ on a projective variety $V$ defined over $k$,
we can associate to each place $v\in M_k$ a function
\[
\lambda_{Y,v}: V\setminus \supp (Y)\to \mathbb R.
\]
Intuitively, for each $P\in V$ and $v\in M_k$, we think of
\[
\lambda_{Y,v}( P)=-\log(v\text{-adic distance from $P$ to $Y$}).
\]
To describe $\lambda_{Y,v}$ more precisely, we use (see for example \cite[Lemma 2.5.2]{Sil}) that for a closed subscheme $Y$ of $V$, there exist effective divisors $D_1,\hdots,D_r$ such that \( Y=\cap D_i \).  Then, the function $\lambda_{Y,v}$ can be described as follows:

\begin{def-thm}[{\cite[Lemma 2.5.2]{vojta_lect}, \cite[Theorem 2.1 (d)(h)]{Sil}}] \label{weil}Let $k$ be a number field, and $M_k$ be the set of places on $k$.  Let $V$ be a projective variety over $k$ and let $Y = \cap D_i \subset V$ be a closed subscheme of $V$.
We define the (local) Weil function for $Y$ with respect to $v\in M_k$ as
\begin{align}\label{WeilYdef}
\lambda_{Y, v}=\min_i \{\lambda_{D_i, v}\},
\end{align}
This is independent of the choices of the $D_i$'s up to an $M_k$-constant, and satisfies
\[
\lambda_{Y_1,v}(P) \le \lambda_{Y_2,v}(P)
\]
up to an $M_k$-constant whenever $Y_1\subseteq Y_2$.  Moreover, if $\pi:  \widetilde V \to V$ is the blowup of $V$ along $Y$ with the exceptional divisor $E$, 
$\lambda_{Y,v}(\pi(P)) = \lambda_{E, v}(P)$ up to an $M_k$-constant as functions on ${ \widetilde V}(k)\setminus E$.
\end{def-thm}

The height function for a closed subscheme $Y$ of $V$ is defined, for $P \in V(k) \setminus Y$, by
\[
h_Y(P):=\sum_{v\in M_k} \lambda_{Y,v}(P).
\]
We also define two related functions for a closed subscheme $Y$ of $V$, depending on a finite set of places $S$ of $k$: the \emph{proximity function} $m_{Y,S}$ and the \emph{counting function} $N_{Y,S}$, for $P\in V(k)\setminus Y$, as
\[
m_{Y,S}(P):=\sum_{v\in S} \lambda_{Y,v}(P)\qquad \text{ and }\qquad N_{Y,S}(P):=\sum_{v\in M_k\setminus S} \lambda_{Y,v}(P)=h_Y(P)-m_{Y,S}(P).
\]

We can now define the notion of $(D,S)$-integral points following Vojta.

\begin{defi}[{\cite[Definition 13.1]{Vojta}}]\label{IPdef}
Let $k$ be a number field and $M_k$ be the set of places on $k$. Let $S\subset M_k$ be a finite subset containing
all Archimedean places.  Let $X$ be a projective variety over $k$, and let $D$ be an effective divisor on $X$.  A set $R\subseteq X(k)\setminus \supp D$ is a $(D,S)$-\emph{integral} set of points if there is a Weil function $\{\lambda_{D,v}\}$ for $D$ and an $M_k$-constant $\{\gamma_v\}$ such that for all $v\notin S$, $\lambda_{D,v}(P)\le \gamma_v$ for all $P\in R$.
\end{defi}
By the uniqueness (up to an $M_k$-constant) of Weil functions for a Cartier divisor $D$ (see \cite[Theorem 9.8 (d)]{Vojta}), one can use a fixed Weil function $\lambda_D$ in Definition \ref{IPdef} (after adjusting $\{\gamma_v\}$).

Finally we recall the definition of arithmetic hyperbolicity.

\begin{defi}\label{arithhyper}
Let $X$ and $D$ as above. We say that $X\setminus D$ is \emph{arithmetically pseudo-hyperbolic} if there exists a proper closed subset $Z\subset X$ such that for any number field $k'\supset k$, every finite set of places $S$ of $k'$ containing the Archimedean places, and every set $R$ of ($k'$-rational) $(D,S)$-integral points on $X$, the set $R\setminus Z$ is finite.  We say that $X\setminus D$ is \emph{arithmetically  hyperbolic} if it is pseudo-arithemtically hyperbolic with $Z=\emptyset$.
\end{defi}

The main tool for relating questions of divisibility between values of polynomials to integrability for points on varieties is established in the following lemma. We state it in terms of local heights since it is more convenient and it admits an explicit analogue using local equations as in \cite{Corvaja:Zannier:2010}.

\begin{lemma}[Compare to {\cite[Lemma 1]{Corvaja:Zannier:2010}}]\label{DivisibilityandIP}
Let $X$ be a projective variety over a number field $k$, and let $S\subset M_k$ be a finite subset containing
all Archimedean places.   Let $D$ be an  effective  Cartier divisor  of $X$ and $W$ be  a closed subscheme  of $X$ such that the codimension of $D\cap W$ is at least 2.  Let $\pi: \widetilde X\to X$ be the blowup along  some closed subscheme of $X$ containing $D \cap W$ such that  $\pi^*D=\widetilde D + \pi^{-1}( D \cap W) $, where   $\widetilde D$ is the strict transform of  $D$.  Let    $R$ be a set of points in $\widetilde X(k)$.  Then the following are equivalent.
\begin{enumerate}
\item[{\rm (i)}] $\lambda_{\widetilde D,v}(P)=0$ up to a $M_k$-constant for  $P\in R$ and  $v\notin S$,
\item[{\rm (ii)}] $\lambda_{D,v}(\pi(P))\le \lambda_{W,v}(\pi(P))$ up to a $M_k$-constant for $P\in R$ and $v\notin S$.
\end{enumerate}
\end{lemma}
\begin{proof}
Let $Y=D \cap W$.  The functorial property of Weil functions implies that
\begin{align}\label{tildeD0}
\lambda_{D,v}(\pi(P))=\lambda_{\pi^*D,v}(P)= \lambda_{\widetilde D,v}(P)+\lambda_{Y,v}(\pi(P))
\end{align} 
up to a $M_k$-constant.
On the other hand, it follows from  \eqref{WeilYdef} that
\begin{align}\label{minWD}\
\lambda_{Y,v}(\pi(P))=\min\{\lambda_{D,v}(\pi(P)), \lambda_{W,v}(\pi(P)) \}
\end{align}  
up to a $M_k$-constant  for  any $v\in M_k$.
Then the equivalence of (i) and (ii) can be easily deduced  from \eqref{tildeD0} and the \eqref{minWD}.
 \end{proof}

 \section{Ru--Vojta Theorem and some basic propositions}\label{sec:RuVojta}
We first recall the following definitions and geometric properties from \cite{ruvojta}.
\begin{defi}
Let
$\mathcal L$ be a big line sheaf  and let $D$ be a nonzero effective Cartier divisor on a projective variety $X$.  We define
\[
\beta_{\mathcal L,D } :=\lim_{N\to\infty}\dfrac{\sum_{m=1}^{\infty} h^0 (V ,  \mathcal L^N(-m D)  ) }{N\cdot h^0(V,\mathcal L^N)}.
\]
If $A$ is a big (Cartier) divisor we let $\beta_{A,D}:= \beta_{\mathcal O(A),D } $. 
\end{defi}

The constant $\beta$ is the crucial ingredient in Ru--Vojta's main Theorem. Before stating it we recall the following definition.

\begin{defi}
Let $D_1,\hdots,D_q$  be effective Cartier divisors on a variety $X$ of dimension $n$.
\begin{enumerate}
\item[{\rm (i)}] We say that $D_1,\hdots,D_q$ are {\it in general position} if for any $I\subset \{1,\dots,q\}$, we have 
\[
 \dim (\cap_{i\in I} \supp  D_i) \le n-\#I  \qquad \text{ with } \dim \emptyset = -\infty.
\]
\item[{\rm (ii)}] We say that $D_1,\hdots,D_q$ {\it intersect properly} if for any $I\subset \{1,\dots,q\}$, $x\in  \cap_{i\in I} \supp  D_i$, and local equations $\phi_i$ for $D_i$ in $x$, the sequence $(\phi_i)_{i\in I}$ is a regular sequence in the local ring $\mathcal O_{X,x}$.
\end{enumerate}
\end{defi}

\begin{remark}\label{gpCM}
If $D_1,\cdots,D_q$ intersect properly, then they are in general position.
By \cite[Theorem 17.4]{Matsumura}, the converse holds if $X$ is Cohen-Macaulay.
\end{remark}

The following is the main arithmetic Theorem of Ru and Vojta.

\begin{theorem} \cite[General Theorem (Arithmetic Part)]{ruvojta}\label{Ru-Vojta} Let $k$ be a number field and $M_k$ be the set of places on $k$. Let $S\subset M_k$ be a finite subset containing the Archimedean places.  Let $X$ be a projective variety defined over $k$.
 Let $D_1,\hdots,D_q$  be effective Cartier divisors intersecting properly on $X$.
 Let  $\mathcal L$ be a big line sheaf on $X$.  Then for any $\e>0$, there exists a   proper Zariski-closed subset 
 $Z\subset X$, independent of $k$ and $S$,  such that
\begin{align}
 \sum_{i=1}^q \beta_{\mathcal L, D_i} m_{D_i,S}(x)\le (1+\e) h_{\mathcal L}(x)
\end{align}
holds for  all but finitely many $x$ in  $X(k)\setminus Z$.
\end{theorem}

We stress that the result is in fact stronger than the original statement, since the exceptional set $Z$ does not depend on $k$ and $S$. This can be obtained by carefully tracing the exceptional sets in the proof with the following version, due to Vojta in \cite{Vojta89}, of Schmidt's subspace theorem, which gives a better control on the exceptional sets.

\begin{theorem}\label{SubspaceVojta} Let $k$ be a number field and $M_k$ be the set of places on $k$. Let $S\subset M_k$ be a finite subset containing the Archimedean places.  Let $H_1,\hdots,H_q$ be hyperplanes in $\mathbb P^n$ defined over $k$ with the corresponding Weil functions $\lambda_{H_1},\hdots,\lambda_{H_q}.$ Then there exist a finite union of hyperplanes $Z$, depending only on $H_1,\hdots,H_q$ (and not on $k$ or $S$), such that for any $\epsilon>0$,
\begin{align}
\sum_{v\in S}\max_{I}\sum_{i\in I}  \lambda_{H_i, v}(P)\le (n+1+\e) h(P)
\end{align}
holds for  all but finitely many points $P$ in  $\mathbb P^n(k)\setminus Z$, where the maximum is taken over subsets $\{1,\hdots,q\}$ such that the linear forms defining $H_i$ for $i\in I$ are linearly independent.
\end{theorem}



We end this section with a useful lemma about local height functions. 
\begin{lemma}\cite[Lemma 5.2]{WaYa}\label{subgeneralweil}
Let $D_1,\cdots,D_q$  be effective divisors of a projective variety $V$  of dimension $n$, defined over $k$, in  general position.
Then
\begin{align}\label{l-inequality}
\sum_{i=1}^q \l_{D_i,v}(P)=  \max_I \sum_{j\in I}  \l_{D_{j},v}(P),
\end{align}
up to a $M_k$- constant,
where $v\in M_k$, $I$ runs over all index subsets of $\{1,\cdots,q\}$ with $n$ elements for all $x\in V(k)$.
\end{lemma}


%
%
 \section{Proof of Theorem \ref{keyprojective} and Theorem \ref{IP'}}\label{sec:proofofKey}
In this section we will prove Theorem \ref{keyprojective} and Theorem \ref{IP'}. These will be obtained as a consequence of the following more general statement. From now on, we denote by $k$ a number field, and by $S$ a finite set of places of $k$.

  \begin{theorem}\label{key}
 Let $V$ be a Cohen--Macaulay projective variety   of dimension $n$  defined over $k$.  Let $D_0,D_1, \dots, D_{r} $, $r\ge n+1$, be    effective  Cartier divisors of $V$  defined over $k$  in general position.   Suppose that there exist an ample Cartier divisor $A$ on $V$ and positive integers $d_i$ such that $D_i\equiv d_iA$ and $d_i\ge d_0$ for all $0\le i\le r$.  
  Then there exists a proper Zariski closed subset $Z$ of $V$, independent of $k$ and $S$,  such that  for any   $M_k$ constant $\{\gamma_v\}$, 
there are only  finitely many $P\in V(k)\setminus Z$ such that 
 the following holds.
\begin{enumerate}
\item[{\rm (i)}]
$r\ge 2n+1$ and $\frac1{d_i}\lambda_{D_i,v}(P)\le \frac1{d_0}\lambda_{D_0,v}(P)+\gamma_v$ for $v\notin S$ and   $1\le i\le r$; or
\item[{\rm (ii)}]
$r\ge n+2$ and $\sum_{i=1}^r \frac1{d_i}\lambda_{D_i,v}(P)\le \frac1{d_0}\lambda_{D_0,v}(P)+\gamma_v$  for $v\notin S$.
\end{enumerate}
\end{theorem}
Here, $\equiv$ denotes numerical equivalence of divisors, and $\lambda_{D_i,v}$ is a Weil function of $D_i$ at $v$.

The following theorem can be deduced from Theorem \ref{key} using Lemma \ref{DivisibilityandIP}.
 
\begin{theorem}\label{IP}
 Let $V$ be a
  Cohen--Macaulay projective variety   of dimension $n$  defined over $k$.  Let $D_0,D_1, \dots, D_{r} $, $r\ge 2n+1$, be  effective  Cartier divisors of $V$  defined over $k$  in general position.   Suppose that there exist an ample Cartier divisor $A$ on $V$ and positive integers $d_i$ such that $D_i\equiv d_iA$ 
  for all $i$.    Let $\pi: \widetilde  V\to V$ be the blowup long the union of subschemes $D_i\cap D_0$, $1\le i\le r$, and let $\widetilde D_i$ be the strict transform of  $D_i$.  If $D=\widetilde D_1+\cdots+\widetilde D_r$, then $ \widetilde  V\setminus D$ is arithmetically pseudo-hyperbolic.
\end{theorem}
It is clear that Theorem \ref{IP'} is a direct consequence of Theorem \ref{IP}.  We now show that Theorem \ref{key} implies Theorem  \ref{keyprojective}.
 
  \begin{proof}[Proof of Theorem  \ref{keyprojective}]
 Let $D_i:=[F_i=0]$ for $1\le i\le r$, and $D_0=[G=0]$.  
 Recall the following standard local Weil function for  $D_i$
\[
\lambda_{D_i,v}(P):=-\log\frac{|F_i(x_0,\hdots,x_n)|_v}{\max\{|x_0|_v^{d_i},\hdots,|x_n|_v^{d_i}\}},
\]
where  $P=[x_0:\cdots:x_n]\in\mathbb P^n(k)\setminus D_i$, $F_0 = G$ and   $d_i=\deg F_i$, $0\le i\le r$.
Since the coefficients of $F_i$ and $G$ are in $\mathcal O_S$, for integral points $P=(x_0,\hdots,x_n)\in\mathcal O_S^{n+1}$, 
the condition that   $F_i(x_0,\hdots,x_n)$ divides $G(x_0,\hdots,x_n)$  in the ring $\mathcal O_S$ 
implies that $|G(x_0,\hdots,x_n) |_v\le |F_i(x_0,\hdots,x_n) |_v\le 1$ for $v\notin S$.
Then    
  $|G(x_0,\hdots,x_n)^{d_i}|_v\le |F_i(x_0,\hdots,x_n)^{d_0}|_v$ for $v\notin S$ as $d_i\ge d_0$, and 
 hence for $v \notin S$,  
\begin{align*}
\frac 1{d_i}\lambda_{D_i,v}(P)-\dfrac {1}{d_0}\lambda_{D_0,v}(P)= -\dfrac{1}{d_0d_i}\log\left|\dfrac{ F_i(x_0,\hdots,x_n)^{d_0} }{ G (x_0,\hdots,x_n)^{d_i}  }\right|_v \le 0  
\end{align*}
Therefore,   Theorem  \ref{keyprojective} (i) is a consequence of Theorem \ref{key} (i).  The proof for (ii) is the same.
 \end{proof}

\subsection{Basic properties and one technical lemma}
 We will recall some basic results and one technical lemma from \cite{WaYa}.
 We start with \cite[Proposition 2.4]{WaYa}, which  is  an immediate consequence
of \cite[Theorem B.3.2.(f)]{HinSil}.
\begin{proposition}\label{compareheight}
Let $X$ be a projective variety defined over $k$, and $A$ be an ample Cartier divisor on $X$ defined over $k$.
Let $D$ be a Cartier divisor $D$ defined over $k$ with $D\equiv A$.
Let $\epsilon>0$.
Then there exists a constant $c_{\epsilon}$ such that for all $P \in X(k)$
\[
(1-\epsilon) h_A(P)-c_{\epsilon} \le h_D(P)\le (1+\epsilon) h_A(P)+c_{\epsilon}.
\] 
\end{proposition}
 
The following theorem  is a reformulation of   \cite[Theorem 3.2]{levin_duke}  by applying    Proposition \ref{compareheight}. 
\begin{theorem}\label{LevinDuke}
Let $X$ be a projective variety of dimension $n$ defined over $k$. Let $D_{1},\hdots,D_{q}$ be effective Cartier divisors on $X$, defined over $k$, in general position.  Suppose that there exists an ample Cartier divisor $A$ on $X$ and positive integer $d_{i}$ such that $D_{i}\equiv d_{i}A$ for all $i$ and all $v\in S$.  Let $\epsilon>0$.  Then there exists a proper Zariski-closed subset $Z\subset X$, independent of $S$ and $k$, such that for all points $P\in X(k)\setminus Z$,
\[
\sum_{i=1}^q\frac1{d_{i}}m_{D_{i}, S}(P)>(q-n-1-\epsilon)h_A(P).
\]
\end{theorem}

 The following proposition follows from \cite[Proposition 5.5]{Kovacs}.   
\begin{proposition}\label{CohenMacaulay}
Let $X$ be a Cohen-Macaulay scheme over $k$ and $Y\subset X$ be a locally complete intersection subscheme.  Let  $\pi: \widetilde X \mapsto X$ be the blowup of $X$ along $Y$.  Then $ \widetilde X$ is a Cohen-Macaulay scheme.  Moreover, if $Z$ is an irreducible subscheme of $Y$,
\[
\dim \pi^{-1}(Z)=\dim Z+ \codim Y-1.
\]
\end{proposition}

 Finally, we need the technical lemma \cite[Lemma 4.7]{WaYa}.
 \begin{lemma}\label{countinglambda}
Let $V$ be a
projective variety of dimension $n$.
Let $D_1, \dots, D_{n+1} $ be    effective  Cartier divisors of $V$ defined over $k$ in general position.   Suppose that there exists an ample Cartier divisor $A$ on $V$   such that $D_i\equiv  A$ for all $1\le i\le n+1$.   Let $Y$ be a   closed subscheme of $V$ of codimension at least 2.    Let $\pi: \widetilde V\to V$ be the blowup along $Y$, and $E$ be the exceptional divisor.  Let $D:=D_1+\cdots +D_{n+1}$.  Then, for all sufficiently large
$\ell$,
  $\calL = \calO(\ell \pi^* D- E)$ is ample and
\begin{align*}
 \beta_{\mathcal L,\pi^*{D_i} }^{-1} \le \frac  1{\ell}\left(1+  O\left(\frac1{\ell^{2}}\right)\right) \le \frac 1\ell \left(1 + \frac 1{\ell\sqrt \ell}\right).
\end{align*}
\end{lemma}

\subsection{Proof of Theorem \ref{key}}
We begin with the following proposition on general position for pullbacks.
  \begin{proposition}\label{basicintersection}
Let $V$ be a Cohen--Macaulay projective variety, and let $D_0,D_1, \dots, D_{r} $  be ample effective Cartier divisors of $V$ in general position. Let $Y_i=D_i\cap D_0$, and let $\pi: \widetilde V\to V$ be the blowup along $Y$, where $Y=\cup_{i=1}^r Y_i$. Finally let $E=E_1+\hdots+E_r$ be the exceptional divisor of $\pi$.  Then, the following holds:
\begin{enumerate}
\item[{\rm(i)}] $\pi^* D_i=\widetilde D_i+E_i$  for each $1\le i\le r$, where $\widetilde D_i$ is the strict transform of $D_i$.
\item[{\rm(ii)}] $ \pi^*  D_1,\dots, \pi^*  D_{r}$ are in general position.
\end{enumerate}
\end{proposition}
\begin{proof}
Since $D_0,\dots,D_r$ are in general position, for every $i \neq j$ the intersection $D_0 \cap D_i \cap D_j$ has codimension at least 3, which implies $(i)$.

To show (ii), we first note that  if $r\ge n$, then the  intersection of any $n+1$ of $ \pi^*  D_i$, $0\le i\le r$,  is empty since  $D_0,D_1, \dots, D_{r} $ are in general position.
Next, let $I\subset \{1,\dots,r\}$ with $\# I\le n$.
We claim that $\dim (\cap_{i\in I} \supp  \pi^*  D_i)\le n-\#I $.
Let $W$ be an irreducible component of $\cap_{i\in I} \supp  \pi^*  D_i$.
If $\pi(W)\subset Y$, then $\pi(W)$ is a subset of $(\cap_{i\in I}  D_i)\cap D_0$ and hence $\dim \pi(W)<n-\# I$.   Then $\dim W\le n-\# I$ by Proposition \ref{CohenMacaulay}.
It remains to consider when $\pi(W)$ is not a subset of $Y$, which   implies that 
$W\setminus \pi^{-1}(Y)$ is not empty and is contained 
  in $\cap_{i\in I} \supp  \widetilde D_i\setminus \pi^{-1}(Y)$.
  Since $(\cap_{i\in I}  \widetilde {D_i})\setminus \pi^{-1}(Y)$ and $(\cap_{i\in I} D_i)   \setminus Y$ are isomorphic, this shows that $\dim W\le n-\# I$.
 \end{proof}

 We can now prove Theorem \ref{key}.

\begin{proof}[Proof of Theorem \ref{key}]
Let $c$ be the least common multiple of $d_0,d_1,\hdots,d_r$.  Let $A_0=cA$, $D_i':=\frac c{d_i} D_i\equiv A_0$, for $0\le i\le r$.  
For $P\in V(k)$ satisfying (i) and $v\notin S$, we have
\[
\lambda_{D_i',v}(P)=\frac c{d_i} \lambda_{D_i,v}(P)\le \frac {d_0}{d_i} \lambda_{D_0',v}(P)\le   \lambda_{D_0',v}(P),
\]
up to a $M_k$ constant  since $d_i\ge d_0$.
Similarly, if   $P\in V(k)$ satisfies (ii), then 
\[
\sum_{i=1}^r \lambda_{D_i',v}(P)=\sum_{i=1}^r\frac c{d_i} \lambda_{D_i,v}(P)\le \frac {c}{d_0} \lambda_{D_0,v}(P)=   \lambda_{D_0',v}(P) 
\]
up to a $M_k$ constant.
Therefore, by replacing $D_i$ by $D_i'$, $0\le i\le r$,  and $A$ by $A_0$, we may assume that $D_i\equiv A$ for each $0\le i\le r$ and replace  (i) by 
\begin{align}\label{i}
\lambda_{D_i,v}(P) \le   \lambda_{D_0,v}(P)+\gamma_v, \quad \text{for } 1\le i\le r;
\end{align}
and, when $v \notin S$, replace  (ii)  by 
\begin{align}\label{ii}
\sum_{i=1}^r \lambda_{D_i,v}(P)\le \lambda_{D_0,v}(P)+\gamma_v.
\end{align}
  
Let $Y_i=D_i\cap D_0$ and 
 $Y=\cup_{i=1}^r Y_i$.   Since $D_0$ is in general position with each $D_i$, $1\le i\le r$,   $D_0$ and each $D_i$ intersect properly by Remark \ref{gpCM}.  Hence, $Y$ is a local complete intersection.   
Let $\pi: \widetilde V\to V$ be the blowup along $Y$, and $E=E_1+\hdots+E_r$, $E_i=\pi^{-1}(Y_i)$, be the exceptional divisors.  Then by Proposition \ref{basicintersection}, $\pi^* D_i=\widetilde D_i+E_i$  for each $1\le i\le r$.
Furthermore, since  $Y$ is a local complete intersection, by Proposition \ref{CohenMacaulay}, $ \widetilde V$ is Cohen-Macaulay  and hence by Proposition \ref{basicintersection}, $ \pi^* D_1,\dots, \pi^* D_{r}$ intersect properly.
 Let $\ell$ be a fixed sufficiently large   integer  such that the line sheaf $\mathcal L = \mathcal O(\ell(n+1) \pi^*A- E)$ is ample and Lemma \ref{countinglambda} holds true, i.e.
\begin{align}
 \beta_{\mathcal L,\pi^*{D_i} }^{-1}   \le \frac  1{\ell}\left(1+  \frac1{\ell \sqrt \ell}\right).
\end{align}
Let $\epsilon' = \ell^{-5/2}$.
Theorem \ref{Ru-Vojta} applied with $\epsilon'$, $ \widetilde V$, $\mathcal L$ and  $\pi^* D_i=  \widetilde{D_i}+E_i$ (for $1\le i\le  r$), gives a proper Zariski closed subset $\widetilde Z\subset \widetilde V$, independent of $k$ and $S$ such that  

\begin{align}
\sum_{i=1}^{r}m_{\pi^*D_i,S}(x) &\le  \left(\frac1\ell(1+  \frac1{\ell\sqrt \ell})+\epsilon'\right)h_{\ell(n+1) \pi^* A- E}(x)\nonumber \\
&\le \left(1 + \frac 2{\ell\sqrt \ell}\right)(n+1) h_{\pi^* A}(x) -  \frac 1\ell  h_E(x) \label{eq:fromruvojta}
\end{align}
holds for  all $x$ outside the proper Zariski-closed subset $\widetilde Z $ of $ \widetilde V(k)$.
By the functorial properties of the local height functions, $h_D=m_{D,S}+N_{D,S}$, and $h_E = h_Y\circ \pi$, we have
\begin{align}\label{eq:fromfunct}
(r-n-1- \frac {2(n+1)}{\ell \sqrt \ell})\cdot h_A(\pi (x))+\frac{1 }{\ell}h_Y(\pi(x))\le    \sum_{i=1}^{r}N_{D_i,S}(\pi(x))
\end{align}
holds for  all $\pi (x)$ outside the proper Zariski-closed subset $Z=\pi(\widetilde Z)$ of $V(k)$. 

For all $P=\pi(x)\in V(k)$ such that  $r\ge 2n+1$ and \eqref{i} holds, i.e. $\lambda_{D_i,v}(P)\le \lambda_{D_0,v}(P)+\gamma_v$ for each $1\le i\le r$, we have
\begin{align}\label{countinglower}
h_Y(P)&\ge N_{Y,S}(P)=\sum_{i=1}^r \sum_{v\notin S}\min\{\lambda_{D_i,v}(P),\lambda_{D_0,v}(P)\}\cr
&=\sum_{i=1}^r N_{D_i,S}(P)-\sum_{i=1}^r\sum_{v\notin S}\max\{0,\gamma_v\}.
\end{align}
Furthermore, it follows from Lemma \ref{subgeneralweil} and Proposition \ref{compareheight} with $\epsilon=\frac1{\ell^2}$ that for all $P\in V(k)$,
\begin{align}\label{counting}
\sum_{i=1}^{r}N_{D_i,S}(P)\le nN_{D_0,S}(P)+O(1)\le (n+\frac1{\ell^2})h_A(P)+O(1).
\end{align}

Apply Theorem \ref{LevinDuke} with $\epsilon=\frac1{\ell^2}$, then there exists a proper  Zariski-closed subset $Z'$ of $ V(k)$, independent of $S$ and $k$, such that, for  all $P\in V(k)\setminus Z'$,
\begin{align}\label{countinglower2}
\sum_{i=1}^r N_{D_i,S}(P) 
&\ge (r-n-1-\frac1{\ell^2})h_A(P).
\end{align}

We now use  \eqref{counting} to get an upper bound for the right hand side of  \eqref{eq:fromfunct} and use \eqref{countinglower}  and \eqref{countinglower2} for the left hand side.  Then we have that
\begin{align}\label{heightbound1}
 (r-2n-1+\frac 1\ell- \frac {2(n+1)}{\ell \sqrt \ell}-\frac2{\ell^2})\cdot h_A(\pi (x)) \le    O(1)
\end{align}
holds for  all  but finitely many  $P\in V(k)$ outside $Z\cup Z'$.
Since $A$ is ample, $r\ge 2n+1$, and  $\frac 1\ell- \frac {2(n+1)}{\ell \sqrt \ell}-\frac2{\ell^2}>0$,  there are only finitely many  $P\in V(k)$ such that \eqref{heightbound1}   holds.  This shows (i).

We are left considering when $r\ge n+2$ and \eqref{ii} holds.
In this case, we have similarly to \eqref{counting}
\begin{align}\label{counting2}
\sum_{i=1}^{r}N_{D_i,S}(P)\le N_{D_0,S}(P)+O(1)\le (1+\frac1{\ell^2})h_A(P)+O(1) 
\end{align}
for all $P\in V(k)$.  
Together with \eqref{eq:fromfunct},  \eqref{countinglower} and \eqref{countinglower2}, we have that
\begin{align}\label{heightbound}
 (r-n-2+\frac 1\ell- \frac {2(n+1)}{\ell \sqrt \ell}-\frac2{\ell^2})\cdot h_A(\pi (x)) \le    O(1)
\end{align}
holds for  all but finitely many  $P\in V(k)$ outside a proper Zariski-closed $Z\cup Z'$.  
Since $A$ is ample, $r\ge n+2$, and  $\frac 1\ell- \frac {2(n+1)}{\ell \sqrt \ell}-\frac2{\ell^2}>0$, this implies (ii).

\end{proof}

\begin{proof}[Proof of Theorem \ref{IP}]
  Since the property of arithmetically pseudo-hyperbolic is independent of the multiplicity of the divisors, we may assume that there exists a positive constant such that 
$D_i\equiv d A$ for $0\le i\le r$.
 
Let $Y_i=D_i\cap D_0$, $Y=\cup_{i=1}^r Y_i$, and  $\pi: \widetilde V\to  V$ be the blowup along $Y$, and $E=E_1+\hdots+E_r$, $E_i=\pi^{-1}(Y_i)$, be the exceptional divisors.
By Proposition \ref{basicintersection}, $\pi^* D_i=\widetilde D_i+E_i$.  
 Let $R$ be a set of  $(D,S)$-integral points, where  $D=\widetilde D_1+\cdots+\widetilde D_r$.
Then there exists $M_k$-constant $\{\gamma_v\}$ such that for all $v\notin S$, $\lambda_{D,v}(P)\le \gamma_v$ for all $P\in R$.
By Lemma \ref{DivisibilityandIP}, we have for each $1\le i\le r$
\begin{align}\label{comparelambda}
\lambda_{D_i,v}(\pi(P))\le \lambda_{D_0,v}(\pi(P))\quad\text{up to a $M_k$-constant for $P\in R$ and $v\notin S$.}
\end{align} 
Since $r\ge 2n+1$, Theorem \ref{key} (i) implies that there exists a proper Zariski closed subset $Z$ of $V$, independent of $k$ and $S$,  such that   there are only  finitely many $\pi(P)\in V(k)\setminus Z$, i.e. $P\notin\pi^{-1}(Z)$, satisfying \eqref{comparelambda}.    
Since the choice of $Z$ is independent of the $M_k$-constant, it implies that $\tilde V\setminus D$ is arithmetically pseudo-hyperbolic.
\end{proof}

\section{Proof of Theorem \ref{proposition7}}\label{sec:prop7}
In this section we prove Theorem \ref{proposition7}. The main technical result is a computation of the constant $\beta$. To this end we generalize some construction of Autissier removing some hypotheses.

\subsection{Background  Results and computing \texorpdfstring{$\beta$}{b}}
We start by recalling some basic properties on global sections of line bundles, and refer to \cite[Section 7.3]{levin_annal} for further references and proofs. 

\begin{lemma}
\label{nef2}
Suppose $D$ is a nef divisor on a nonsingular projective variety $X$.  Let $n=\dim X$.  Then $h^0(X,\mathcal{O}(ND))=(D^n/n!)N^n+O(N^{n-1}).$  In particular, $D^n>0$ if and only if $D$ is big. 
\end{lemma}

We will also make use of two basic exact sequences.(See \cite[Lemma 7.7]{levin_annal}.)
 
\begin{lemma}
\label{exact}
Let $D$ be an effective divisor on a projective variety $X$ with inclusion map $i:D \to X$.  Let $\mathcal{L}$ be an invertible sheaf on $X$.  Then we have exact sequences
\begin{align*}
&0 \to \mathcal{L}\otimes\mathcal{O}(-D) \to \mathcal{L} \to i_{*}(i^*\mathcal{L}) \to 0,\\
&0 \to H^0(X,\mathcal{L}\otimes \mathcal{O}(-D))\to H^0(X,\mathcal{L}) \to H^0(D,i^*\mathcal{L}).
\end{align*}
\end{lemma}

\begin{lemma}[{\cite[Lemma 7.9]{levin_annal}}]
\label{countingh0}
Let $X$ be a nonsingular projective variety of dimension $n$. Let $D$ and $E$ be any divisor on $X$, and let $F$ be a nef divisor on $X$.  Then we have  
\begin{align*}
h^0(X,\mathcal{O}(ND+E-mF))\le h^0(X,\mathcal{O}(ND))+O(N^{n-1})\quad\text{for all } m, N\ge 0,
\end{align*}
where the implied constant is independent of $m$ and $N$.
\end{lemma}

We will use the following lemma and its corollary, which are modification of  {\cite[Lemma 4.2]{aut}} and {\cite[Corollary 4.3]{aut}} where we weaken the original hypothesis on $B$.
\begin{lemma}\label{lem_aut1_4_2}
Let $X$ be a  nonsingular projective variety of dimension $n\ge 2$.  Let $B$ be a nonsingular subvariety of $X$ of codimension $1$ that is also a nef Cartier divisor.  Let $A$ be a nef Cartier divisor on $X$ such that $A-B$ is also nef. Let $\delta>0$ be a positive real number. Then, for any positive integers  $N$ and $m$  with $1 \leq m\leq \delta N$, we have
\begin{eqnarray}\label{h0AB}
h^0(X,\mathcal{O}(NA-mB))&\ge& \frac{A^n}{n!}N^n 
- \frac{A^{n-1}.B}{(n-1)!}N^{n-1}m \cr
&~&+ \frac{(n-1)A^{n-2}.B^2}{n!} N^{n-2}\min\{m^2, N^2\} -O(N^{n-1}),
\end{eqnarray}
where the implicit constant depends on $\delta$.
\end{lemma} 

\begin{proof}[Proof of Lemma \ref{lem_aut1_4_2}]
We will follow the proof of  {\cite[Lemma 4.2]{aut}} with necessary modification.
We first note when $m \leq N$, \eqref{h0AB} follows from the proof in {\cite[Lemma 4.2]{aut}}, since this part of proof only need the assumption that $A$, $B$ and $A-B$ are nef.

For the case that $m>N$, we let $N\le j\le m$.
Let $i:B\to X$ be the inclusion map.  From Lemma \ref{exact}, we have an exact sequence
\begin{align*}
0\to H^0(X,\mathcal{O}(NA-(j+1)B)) \to H^0&(X,\mathcal{O}(NA-jB)) 
 \to H^0(B,i^*\mathcal{O}(NA-jB)).
\end{align*}
Therefore, we have 
\begin{align*}
h^0(X,\mathcal{O}(NA-(j+1)B))\ge h^0&(X,\mathcal{O}(NA-jB))-h^0(B,i^*\mathcal{O}(NA-jB)).
\end{align*}
Hence,
\begin{align}\label{NAB2}
h^0(X,\mathcal{O}(NA-mB))\ge h^0&(X,\mathcal{O}(NA-NB))-\sum_{j=N}^{m-1}h^0(B,i^*\mathcal{O}(NA-jB)).
\end{align}
 Since $B$ is a nef divisor on $X$, $i^*\calO(B)$ is nef.  Applying Lemma \ref{countingh0} to $B$, which is a non-singular subvariety of $X$, and the divisors corresponding to $i^*\mathcal{O}(A)$ and $i^*\mathcal{O}(B)$, we have

\begin{align}\label{NAB3}
 h^0(B,i^*\mathcal{O}(NA-jB))\le h^0(B,i^*\mathcal{O}(NA))+O(N^{n-2})=\frac{A^{n-1}B}{(n-1)!}N^{n-1}+O(N^{n-2}). 
\end{align}
Then, from \eqref{NAB2}, \eqref{NAB3}, Lemma \ref{nef2}, and the estimate of  $h^0 (X,\mathcal{O}(NA-NB))$ in the first case, it follows that
\begin{align*}
h^0(X,\mathcal{O}(NA-mB))&\ge h^0 (X,\mathcal{O}(NA-NB))-(m-N)\frac{A^{n-1}B}{(n-1)!}N^{n-1}-O(N^{n-2})\cr
&\ge \dfrac{A^n}{n!}N^n 
- \dfrac{A^{n-1}.B}{(n-1)!}N^{n-1}m  + \dfrac{(n-1)A^{n-2}.B^2}{n!} N^{n}  -O(N^{n-1}).
\end{align*} 
This shows \eqref{h0AB} for the case that $m>N$.
\end{proof}

We use Lemma \ref{lem_aut1_4_2} to obtain a lower bound on the $\beta$ constant in terms of intersection numbers.
\begin{corollary}\label{beta}
Let $X$ be a nonsingular projective variety of dimension $n\ge 2$.  Let $B$ be a nonsingular subvariety of $X$ of codimension $1$ that is also a nef Cartier divisor on $X$.  Let $A$ be a big and nef Cartier divisor on $X$ such that $A-B$ is  nef. Then
\[
\beta_{A, B} \ge  \frac { A^n}{2n A^{n-1}.B}  +\frac{(n-1)A^{n-2}.B^2}{ A^n}g(\dfrac{A^n}{n A^{n-1}.B}),
\]
  where $g: \mathbb R^+\rightarrow \mathbb R^+$ is the function given by $g(x)={x^3 /3}$ if $x\leq 1$ and $g(x)=x-{2 / 3}$ for $x\ge 1$.  
\end{corollary}
\begin{proof}
Let 
\[
  b=\dfrac{A^n}{n A^{n-1}.B}\qquad \text{ and } \qquad a=(n-1)A^{n-2}.B^2.
  \]
For $N$ sufficiently large and such that $bN$ is an integer,  Lemma \ref{lem_aut1_4_2} implies that
\begin{equation*}
 \begin{split}
  &\sum_{m=1}^{\infty} h^0(NA-m B) \label{aut1} \\
  &\quad \ge \sum_{m=1}^{ b N }
    \left( \dfrac{A^n}{ n!}N^n 
    - \dfrac{A^{n-1}.B}{ (n-1)!}N^{n-1}m + \dfrac{a}{n!}N^{n-2}\min\{m^2, N^2\}\right)+O(N^n) \\
  &\quad\ge \left( \dfrac{A^n}{n!} b
    - \dfrac{A^{n-1}.B}{(n-1)!}\cdot\dfrac{b^2}{2} + \dfrac{a}{n!}g(b)\right)N^{n+1}+O(N^n) \\
  &\quad= \left(\dfrac{b}{2}+\dfrac{a}{A^n}g(b)\right) A^n \dfrac{N^{n+1}}{n!} +O(N^n). \\
 \end{split}
\end{equation*}
\end{proof}
\subsection{Proof of Theorem \ref{proposition7}}
We apply Corollary \ref{beta} to the setting of Theorem \ref{proposition7}.
\begin{lemma}\label{beta2}
Let $H_1,\hdots,H_{2n}$ be $2n$ hyperplanes in general position on $\mathbb P^n$ and choose $n+1$ points $P_i$ such that $P_i\in H_i$, $1\le i\le n+1$, and $P_i\notin H_j$ if $i\ne j$ for $1\le j\le 2n$.  Let $\pi:  X\to\mathbb P^n$ be the blowup of the $n+1$ points $P_i$, and let $ \widetilde H_i$ be the strict transform of $H_i$. Finally, let $\ell$ be a sufficiently large integer and let $A =\sum_{i=1}^{n+1} \ell \widetilde H_i+\widetilde H_{n+2}.$  Then $A$ is big and nef  and
\begin{align}\label{betai} 
 \beta_{A, \widetilde H_{1}}=\cdots= \beta_{A, \widetilde H_{n+1}}>\frac{(n+1)\ell}{2n}+o(\ell), 
\end{align} 
\begin{align}\label{betan2}
 \beta_{A, \widetilde H_{n+2}}=\cdots= \beta_{A, \widetilde H_{2n}}>\frac{(n+1)\ell}{2n}-\frac{\ell}{2n(n+1)^{n-2}}+o(\ell).  
 \end{align}

\end{lemma}
\begin{proof}
Let $\pi:  X\to\mathbb P^n$ be the blowup of the points $P_i$, as in the hypotheses.
Let $E_i=\pi^{-1}(P_i)$, be the exceptional divisors.
 Then 
\begin{align}\label{Hi2}
 \pi^*H_i=\left\{
 \begin{array}{lr}
 \widetilde H_i+E_i, &\text{for } 1\le i\le  n+1\\
 \widetilde H_{i}, &\text{for } n+2\le i\le  2n.
 \end{array}\right.
\end{align}

Moreover,
\begin{align}\label{A2}
A:= \sum_{i=1}^{n+1} \ell \widetilde H_i+\widetilde H_{n+2}\sim (\ell(n+1)+1)\pi^*H-\ell E,
\end{align}
where $E=E_1+\cdots+E_{n+1}$ and $H$ is a (generic) hyperplane in $\mathbb P^n$.
Then
\begin{align}\label{An}
 A^n=(( \ell(n+1)+1)\pi^* H-\ell E)^n 
=\big((n+1)^n-(n+1)\big) \ell^n+O(\ell^{n-1}).
\end{align}
We now show that $\widetilde H_i$ is   nef for $1\le i\le 2n$.  
Let $C$ be an irreducible curve on $X$.
Then
\begin{align}\label{intersectnumber}
 \widetilde H_i.C=\left\{
 \begin{array}{lr}
  \pi^*H.C-E_i.C, &\text{for } 1\le i\le  n+1\\
 \pi^*H.C , &\text{for } n+2\le i\le  2n.
 \end{array}\right\}.
\end{align}
 If $C\subset E$, then $C$ is contained in exactly one of the $E_i$. Let us assume that $C\subset E_j$ for some $1\le j\le n+1$.  Then $E_j.C<0$, $E_i.C=0$ if $i\ne j$ and
 $\pi^*H.C=H.\pi_*C=0$.  Hence, $\widetilde H_j.C >0$   and $\widetilde H_i.C =0$ for $1\le i\ne j\le 2n$. 
Otherwise, $\pi_*(C)$ is   a curve in $\mathbb P^n$.   Then for $1\le i\le n+1$,
\begin{align*}
 \widetilde H_i.C=\pi^*H.C-E_i.C&= H.\pi_*C- {\rm multi}_{P_i}\pi_*(C)\cr
&= \deg \pi_*C-\ {\rm multi}_{P_i}\pi_*(C)\ge 0,
\end{align*}
since ${\rm multi}_{P_i}\pi_*(C)\le \deg \pi_*C;$ 
and for $n+1\le i\le 2n$,
\begin{align*}
 \widetilde H_i.C=\pi^*H.C = H.\pi_*C = \deg \pi_*C>0.
\end{align*}
 Therefore,   $ \widetilde H_i$ is nef for each $1\le i\le 2n$  and hence $A$ is also nef. Since $A$ is nef and $A^n>0$ by \eqref{An} as $n\ge 2$, Lemma \ref{nef2} implies that $A$ is big.

Our assertions \eqref{betai} and \eqref{betan2}  can be easily obtained from 
  Corollary \ref{beta} by noting that for $i \leq n+2$, $A-\widetilde H_i$, is still nef and by \eqref{An} and estimating the following intersection numbers.

 \begin{align*} 
 A^{n-1}.\widetilde H_i&=(( \ell(n+1)+1)\pi^* H-\ell E)^{n-1}. (\pi^* H- E_i)\cr
 & = ((n+1)^{n-1}-1)\ell^{n-1}+O(\ell^{n-2}) \quad\text{for  }1\le i\le n+1,
\end{align*} 
and 
 \begin{align*} 
A^{n-1}.\widetilde H_i&=(( \ell(n+1)+1)\pi^* H-\ell E)^{n-1}.  \pi^* H \cr
 & =  (n+1)^{n-1} \ell^{n-1}+O(\ell^{n-2})  \quad\text{for  } n+2\le i\le 2n.
\end{align*} 
\end{proof}

We can now prove Theorem \ref{proposition7}. 

\begin{proof}[Proof of Theorem \ref{proposition7}]
 Let $\pi:  X\to\mathbb P^n$ be the blowup of the $n+1$ points $P_i$, $1\le i\le n+1$, such that $P_i\in H_i$, and $P_i\notin H_j$ if $j\ne i$.  Let $E_i=\pi^{-1}(P_i)$, $1\le i\le n+1$, be the exceptional divisors.
 We note that  $X$ is smooth and the strict transforms $\widetilde H_1,\hdots,\widetilde H_{2n}$ are in general position.

  Let  $\ell $ be a sufficiently large  integer to be determined later.  
  Let $A =\sum_{i=1}^{n+1} \ell \widetilde H_i+\widetilde H_{n+2}\sim (\ell(n+1)+1)\pi^*H-\ell E,$
where $E=E_1+\cdots+E_{n+1}$.
By Lemma \ref{beta2}, $A$ is   big and nef  and there exist  constants $\beta$ and $\widetilde\beta$ such that 
 \begin{align}
 \beta\cdot \ell &= \beta_{A, \widetilde H_{1}}=\cdots=\beta_{A, \widetilde H_{n+1}},\cr
 \widetilde\beta\cdot\ell& = \beta_{A, \widetilde H_{n+2}}=\cdots=\beta_{A, \widetilde H_{2n}}, 
\end{align}
and  
\begin{align}\label{usingRV4} 
\delta:=(n-1)\widetilde\beta+2 \beta-2>\frac{(n-1)}{2n}\left(n-1-\frac1{(n+2)^{n-2}}\right)\ge 0 
\end{align}
since $n\ge 2$.  
Then we let
\begin{align*}
 \epsilon:=\frac{\delta}{4(n+3)}>0. 
 \end{align*}
Applying Theorem \ref{Ru-Vojta} to  $\epsilon $, $ X$, $\mathcal L = \mathcal O(A)$ and  $\widetilde{H_i}$,  $1\le i\le  2n$, there exists a proper Zariski closed subset $Z\subset X$, independent of $k$ and $S$, such that  
\begin{align}\label{usingRV} 
 \beta\sum_{i=1}^{n+1}m_{\widetilde H_i,S}(x) +\widetilde\beta\sum_{i=n+2}^{2n}m_{\widetilde H_{i},S}(x) 
 \le  \left(1+ \epsilon\right)(n+1+\frac1\ell)h_{\pi^*H}(x)-(1+\epsilon) h_E(x) 
\end{align}
holds for  all $x$  in $  X(k)\setminus Z$.
For the set $R$ of $(D,S)$-integral points, with $D=\widetilde H_1+\hdots+\widetilde H_{2n}$,    let $c_R:=\sum_{v\notin S} \max\{0, \gamma_v\}$ , where $\{\gamma_v\} $ is  the $M_k$ constant from Definition \ref{IPdef}.   Then for $x\in R$,  
\begin{align*} 
\sum_{i=1}^{2n} N_{\pi^* H_i,S}(x)-N_{E,S}(x)=\sum_{i=1}^{2n} N_{\widetilde H_i,S}(x) \le c_R,  
 \end{align*} 
and hence 
\begin{align}\label{counting1}
N_{E,S}(x)\ge  \sum_{i=1}^{2n} N_{H_i,S}(\pi(x))-O(1). 
 \end{align} 
 Moreover, since $x \in R$,    
\begin{align}\label{proximity1}
\sum_{i=1}^{n+1}  m_{\widetilde H_i,S}(x)\ge \sum_{i=1}^{n+1} h_{\widetilde H_i}(x)-O(1)=(n+1)h_{\pi^*H}(x)-h_{E}(x)-O(1),
 \end{align} 
 and 
\begin{align}\label{proximity2}
\sum_{i=n+2}^{2n}m_{\widetilde H_{i},S}(x) \ge \sum_{i=n+2}^{2n} h_{\widetilde H_{i}}(x)-O(1)=(n-1)h_{\pi^*H}(x)-O(1).
 \end{align} 
 Using \eqref{proximity1} and \eqref{proximity2}, and assuming $\ell > \frac{1}{\epsilon}$, we can rewrite \eqref{usingRV} as
 \begin{align}\label{usingRV2} 
(1-\beta)h_E(x)\le \big(n+1- (n+1)\beta -(n-1)\widetilde\beta + (n+3) \epsilon\big) h_{\pi^*H}(x)+O(1).
\end{align}

 On the other hand, by Lemma \ref{subgeneralweil} and the fact that $m_{H_i,S}(\pi(x))+N_{H_i,S}(\pi(x))=h(\pi(x))+O(1)$,  we can derive from 
Theorem \ref{SubspaceVojta} that  there exists a finite union of hyperplanes $W$, independent of $k$ and $S$, such that for any $\epsilon'>0$
\begin{align}\label{smtapp}
 \sum_{i=1}^{2n} N_{H_i,S}(\pi(x))\ge (n-1-\epsilon') h(\pi(x))-O(1)
\end{align}
for all but finitely many $\pi(x)$ in $\mathbb P^n(k)\setminus W$.  

Since $h_E(x)\ge N_E(x)$,   we can deduce from  \eqref{counting1} and \eqref{smtapp} that 
for all but finitely many $\pi(x)$ in $\mathbb P^n(k)\setminus W$
 \begin{align}\label{hE}
 h_E(x)\ge  (n-1-\epsilon') h(\pi(x))-O(1).
 \end{align}
 
Then we derive from  \eqref{usingRV4}, \eqref{usingRV2}, and \eqref{hE}  that
  \begin{align*} 
 (\delta-(1-\beta)\epsilon'-(n+3)\epsilon)h(\pi (x)) \le   O(1) 
\end{align*}
 for all but finitely  $x\in R$ outside $Z\cup \pi^*(W)\cup E$.  
 Let $\epsilon'\le \delta/4(1-\beta)$.  Then, by definition of $\epsilon$, 
   \begin{align*} 
\frac{\delta}2 h(\pi (x)) \le   O(1), 
\end{align*}
which can only be satisfied for finitely many $\pi (x)\in \mathbb P^n(k)$.  Therefore,
 there are only finitely many  $x\in R$ outside $Z\cup {\rm Supp } \pi^*(W)\cup {\rm Supp } E$.   
  \end{proof}

\section{Proof of Theorem  \ref{proposition1} and Theorem \ref{Theorem7}}\label{sec:proof_last}
 
 We fix the notation we will use throughout this section.
 Let $n\ge 2$, $q\ge 3n$ be integers. For every index $i\in \mathbb Z/q\mathbb Z$, let $H_i$ be a hyperplane in $\mathbb P^n$ defined over $k$. Suppose that $H_1,\hdots,H_q$  are in general position.  For each  index $i\in \mathbb Z/q\mathbb Z$, let $P_i$ be the intersection point $\cap_{j=0}^{n-1} H_{i+j}$.  Let $\pi:  X\to \mathbb P^n$ be the blow-up over the points $P_1,\hdots, P_q$ and $E_i=\pi^{-1}(P_i)$, $1\le i\le q$, be the exceptional divisors.
 Let $\widetilde H_i\subset  X$ be the corresponding strict transform of $H_i$ and let $D=\widetilde H_1+\cdots+\widetilde H_q$. 
 Since $P_i$ is the intersection point $\cap_{j=0}^{n-1} H_{i+j}$,
 we have 
\begin{align}\label{Hi}
 \pi^*H_i=\widetilde H_i+\sum_{j=i-n+1}^i E_j,
\end{align}
 and
\begin{align}\label{D}
 D=\sum_{i=1}^q \widetilde H_i\sim q\pi^*H-n\sum_{i=1}^q E_i.
\end{align}

\subsection{Key Lemmas}
We collect here the key lemmas for computing the constant $\beta$.
\begin{lemma}\label{DDi}
Let  $n\ge 2$, $q\ge 3n$ and let $D$ and $\widetilde H_i$ be as defined above. Then, for every $1\le i\le q$ and $0 \leq m \leq n$, the divisor $D-m\widetilde H_i$ is nef.
\end{lemma}

\begin{proof}
Recall that  $ \pi: X\to\mathbb P^n$ is the blowup of the points $P_i$, as described above.

It is clear that it suffices to show $D-m\widetilde H_q$ is nef if $q\ge 3n$ and $0\le m\le n$ by rearranging the index.  By \eqref{Hi} and \eqref{D}, we have
\begin{align}\label{Dm}
 D-m\widetilde H_q \sim (q-m)\pi^*H-n \sum_{i=1}^{q-n} E_i-(n -m)\sum_{i=q-n+1}^{q} E_i.
\end{align}
Let $C$ be an irreducible curve on $X$.  If $\pi_*(C)$ is not a curve in $\mathbb P^n$, then $\pi_*(C)=P_i$ for some $i$.  Hence, 
$\pi^*H.C=H.\pi_*C=0$,  $ E_j.C=0 $ for $1\le j\ne i\le q$  and $ E_i.C>0 $.
Therefore, \eqref{Dm} gives $(D-m\widetilde H_q ).C\ge 0$ if $0\le m\le n$.

If $\pi_*(C)$ is  a curve in $\mathbb P^n$ and  $\pi_*(C)$ is not in any of the $H_i$, from \eqref{Dm} we have
\begin{align}\label{intersectnumber3} 
&( D-m\widetilde H_q).C \cr
&=(q -m) H.\pi_*(C)-n \sum_{i=1}^{q-n} {\rm multi}_{P_i}\pi_*(C)-(n -m)\sum_{i=q-n+1}^{q} {\rm multi}_{P_i}\pi_*(C).
\end{align}
It suffices to find $q-m$ hyperplanes passing through $P_1,\hdots,P_q$ with described multiplicity as the equation above.  We note that each $H_i$ contains exactly $n$ points, $P_{i-n+1},\hdots, P_i$,  among the $P_j$'s; and each point $P_i$ is contained in exactly $n$ hyperplanes, $H_i,\hdots,H_{n+i-1}$, among the $H_j$'s.
We first consider the points $P_i$'s contained in  $H_{n},\hdots,H_{q-n}$ and denote these points with multiplicities as a formal sum below.   
\begin{align}\label{countingpoints}
n\sum_{i=n}^{ q- 2n+1 }P_i+\sum_{i=1}^{n-1} i (P_i+P_{q-n+1-i})=n\sum_{i=1}^{q-n}P_i-\sum_{i=1}^{n-1}(n- i )(P_i+P_{q-n+1-i})
\end{align}
Recall that $q\ge 3n$.  The last sum in the right hand side of \eqref{countingpoints} contains $n(n-1)$ points counting multiplicity and the multiplicities of these points range from one to $n-1$. 
Therefore, we may choose $n-1$ hyperplanes $L_1,\hdots,L_{n-1}$ containing these $n(n-1)$ points (counting multiplicity). Then together with \eqref{countingpoints}, we have
\begin{align}\label{intersection1}
(q-n) H.\pi_*(C)= \sum_{i=n}^{q-n} H_i.\pi_*(C)+ \sum_{i=1}^{n-1}L_i.\pi_*(C)\ge n \sum_{i=1}^{q-n} {\rm multi}_{P_i}\pi_*(C).
\end{align}
Finally, since $P_{q-n+1},\hdots,P_q\in H_q$,
we have 
\begin{align}\label{intersection2}
(n -m)H.\pi_*(C)=(n -m) H_q.\pi_*(C)\ge (n -m)\sum_{i=q-n+1}^{q} {\rm multi}_{P_i}\pi_*(C).
\end{align}
Then $(D-m\widetilde H_q).C\ge 0$ if $q\ge 3n$ and $0\le m\le n$
by  \eqref{intersectnumber3}, \eqref{intersection1} and \eqref{intersection2}.

Finally, we   consider the case where $\pi_*(C)$ is contained in some $H_b$, where $1\le b\le q$. 

Suppose that $\pi_*(C)\subset \cap_{t=0}^{a}H_{b-t}$ and   $\pi_*(C)\not\subset H_{b-a-1}$.  Clearly, $0\le a\le n-2$ since the $H_i$ are in general position and $\pi_*(C)$ is a curve.  Then $\pi_*(C)\cap\{P_1,\hdots,P_q\}\subseteq  \{P_{b-n+1},\hdots,P_{b-a}\},$ which is contained in $H_{b-a-1}  \cup H_{b-a+j} $, for $a+1\le j\le n-1$.  Since $\pi_*(C)\subset\cap_{t=0}^{a}H_{b-t}$, it cannot be contained in every $H_{b-a+j}$,  $a+1\le j\le n-1$.  Suppose that $\pi_*(C)$ is not contained in $H_{j_0}$, for some $b-a+1\le j_0\le b-a+n-1$.
Then we have

\begin{align}\label{intersection3}
2H.\pi_*(C)=  (H_{b-a-1}+H_{j_0}).\pi_*(C)\ge   \sum_{i=1}^{q } {\rm multi}_{P_i}\pi_*(C).
\end{align}
Then, by  \eqref{intersectnumber3} since $m\le n$ and $q \geq 3n$, we have $( D-m\widetilde H_q).C\ge(q -2n-m) H.\pi_*(C) \ge 0$. 
\end{proof}

\begin{lemma}\label{beta3}
  Let  $n\ge 2$ and $q\ge 3n$.  Let $D$ and $\widetilde H_i$, be as above.
  Then $D$ is big and
   \[ \beta_{D,\widetilde H_1}=\cdots=\beta_{D,\widetilde H_q}>1.\] 
  \end{lemma}
 \begin{proof}
Since $D$ is nef by Lemma \ref{DDi},  we have 
\[
h^0(X,\mathcal O(ND))=\frac{D^n}{n!}\cdot N^n+O(N^{n-1}) 
\]
by Lemma \ref{nef2}.
It follows from \eqref{D} that
\begin{align}\label{Dn0}
D^n=(q\pi^*H-n\sum_{i=1}^q E_i)^n= q^n-n^nq.
\end{align}
Therefore, $D$ is big if $q^{n-1}> n^n$,  which is satisfied when $n\ge 2$ and $q\ge 3n$.  
 
By the Hirzebruch-Riemann-Roch theorem, adapting the arguments in \cite[Lemma 4.2]{aut}, we obtain
\[
\chi(X;ND- m\widetilde H_i)=\frac 1{n!}(ND- m\widetilde H_i)^n +O(N^{n-1}),
\]
where $\chi(X;\cdot)$ is the Euler characteristic.

Since $D$ and $D-b\widetilde H_i$ are nef for $0\le b\le n$,
$h^{i}(X,\mathcal{O}(ND-m\widetilde H_i)=O(N^{n-i})$
if $m\le nN$ (see e.g. \cite[Theorem 1.4.40]{Laza}). Therefore,
\begin{align}\label{H0} 
h^0(X,\mathcal O(ND-m\widetilde H_i))=\frac{(ND-m\widetilde H_i)^n}{n!}\cdot N^n+O(N^{n-1}),\quad\text{for } m\le n N.
\end{align} 
 
By \eqref{Hi} and \eqref{D}, we can compute
\begin{align}\label{DH} 
 D^{k}.\widetilde H_i^{n-k}=q^{k}-n^{k+1},  \quad\text{for } k\le n-1.
\end{align} 
Then  
\begin{align}\label{m=N}
(ND-m\widetilde H_i)^n&= (q^n-n^nq)N^n+\sum_{k=0}^{n-1} \binom {n}{k}(q^k-n^{k+1})(-1)^{n-k}N^km^{n-k}\cr
&=(qN-m)^n-n(nN-m)^n-n^n(q-n)N^n.
\end{align}
In particular, for $m=nN$ 
\begin{align}\label{m=nN}
(ND-n\widetilde H_i)^n =(q-n)\big((q -n)^{n-1} -n^n \big)N^n\ge 0
\end{align}
since $q\ge 3n$ and $n\ge 2$.
Since \eqref{m=N} is a decreasing function in $m$, the right hand side of \eqref{m=N} is nonnegative for $m\le nN$. 
By \eqref{H0} and  \eqref{m=N}, we have
\begin{align*} 
&n!\sum_{m=0}^{n N} h^0(X,\mathcal{O}(ND-m\widetilde H_i))\cr
&=\sum_{m=0}^{n N} (qN-m)^n-n(nN-m)^n-n^n(q-n)N^n+O(N^{n})\cr
&=\big( q^{n+1}-(q-n)^{n+1}-n^{n+2}  -n^{n+1}(n+1)(q-n) \big)\frac {N^{n+1}}{n+1}+O(N^{n}).
\end{align*}
Together with \eqref{Dn0}, it yields
\[
\beta_{D,\widetilde H_i}\ge \beta:=\frac{q^{n+1}-(q-n)^{n+1}-n^{n+2}  -n^{n+1}(n+1)(q-n)}{(n+1)(q^n-n^nq)}.
\]
We  now show that $ \beta>1$.  Let
\begin{align}\label{f(q)}
f(q):&=(\beta-1)(n+1)(q^n-n^nq)\cr
&=q^{n+1}-(q-n)^{n+1}-n^{n+2}  -n^{n+1}(n+1)(q-n)-(n+1)(q^n-n^nq)\cr
&=  q^{n+1}-(q-n)^{n+1}-(n+1)q^n-(n^2-1) n^n(q-n)+n^{n+1}.
\end{align}

We will need to show that $f(q)>0$ if $q\ge 3n$.
\begin{align*} 
f'(q) &= (n+1)(q^{n}-(q-n)^{n})-(n+1)nq^{n-1}-n^{n+2} +n^n\cr
&= n(n+1)(q^{n-1}+q^{n-2}(q-n)+\hdots+(q-n)^{n-1})-(n+1)nq^{n-1}-n^{n+2} +n^n\cr
&=n(n+1)( q^{n-2}(q-n)+\hdots+(q-n)^{n-1})-n^{n+2} +n^n\cr
&>(n^3-n) (q-n)^{n-1}-n^{n+2}\ge 0\qquad\text{if    $q\ge3n$ and $n\ge 2$}.
\end{align*}

Therefore, it suffices to show that $f(3n)>0$.  By \eqref {f(q)},
 \begin{align*} 
 f(3n)&=n^n\big( (2n-1)\cdot 3^{n}-n\cdot 2^{n+1}  -(2n^2-3)n  \big).
 \end{align*}
 It is easy to check that  $f(3n)>0$ for $n=2$.  We now assume that  $n\ge 3$.
 Then   
  \begin{align*} 
  &(2n-1)\cdot 3^{n}-n\cdot 2^{n+1}  -(2n^2-3)n \cr
  &\ge  n \cdot3^{n}-n\cdot 2^{n+1}+  (n-1) 3^{n} -(3n^2-3)n\cr 
  &\ge  9n\cdot 3^{n-2}-8n\cdot 2^{n-2}+(n-1)( 3^{n}-3n(n+1)) >0.
  \end{align*}
This show that $f(3n)>0$ for $n\ge 3$ as well.
  \end{proof}

 \subsection{Proof of Theorem  \ref{proposition1} and Theorem \ref{Theorem7}}
 \begin{proof}[Proof of Theorem  \ref{proposition1}]
  We note that since $X$ is smooth we need to verify that $\widetilde H_1,\hdots,\widetilde H_q$ are in general position in order to apply Theorem \ref{Ru-Vojta}.  
   Let $W=\widetilde H_1  \cap \widetilde H_2\cap\cdots \cap\widetilde H_i$ (after reindexing) $1\le i\le n$.  Following the proof of Proposition \ref{basicintersection}, it suffices to consider when
$\pi(W) \subset \{P_1,\hdots,P_q\}$.  Since $W$ is irreducible, it implies that   $\pi(W)$ is some $P_j\in  H_1  \cap H_2\cap\cdots \cap H_i$.  Since $H_1,\hdots,H_q$ are hyperplanes in general position, they intersect transversally.  Thus, the codimension of $W$ is $i$.  

Since $n\ge 2$ and $q\ge 3n$, it follows from  Lemma \ref{beta3} that  $D$ is big and
$\beta:=  \beta_{D,\widetilde H_1}=\cdots=\beta_{D,\widetilde H_q}>1$.
Theorem \ref{Ru-Vojta} with $\epsilon=\frac 12 (\beta-1)$ implies that there exists a proper Zariski closed set $Z\subset X$ independent of $k$ and $S$  such that 
\begin{align}
\beta \cdot\sum_{i=1}^q m_{\widetilde H_i,S}(x)\le (1+\epsilon)h_D(x)
\end{align}
for all but finitely many $x\in X(k)\setminus Z$.
Let $R$ be a set of $(D,S)$-integral points.  Then 
\[
\sum_{i=1}^q m_{\widetilde H_i,S}(x)=h_D(x)+O(1),
\]
where the constant depends only on $R$.
Hence,

\begin{align}\label{big}
\frac 12 (\beta-1) h_D(x)=(\beta-1-\epsilon)h_D(x)\le O(1)
\end{align}
for all but finitely many $x\in R$ outside $Z$.
Since $D$ is big, for a given ample divisor $A$, there exists a proper Zariski closed set   $Z'$ of $X$, depending only on $A$ and $D$ such that $h_A(x)\le h_D(x)+O(1)$ for all $x\in X(\bar k)$ outside of $Z'$ (see \cite[Proposition 10.11]{Vojta}).
Therefore, \eqref{big} implies that there are only finitely many  $x\in R$ outside $Z\cup Z'$.
\end{proof}

 \begin{proof}[Proof of Theorem  \ref{Theorem7} ]
 Denote by ${\bf x}:=[x_0:\cdots:x_n]\in\mathbb P^n(k)$. Up to enlarging the set $S$, we can suppose that $x_0,\hdots,x_n$ are $S$-integers and that the ring $\mathcal O_S$ is a unique factorization domain.  Let $H_i=[F_i=0]$ for $1\le i\le q$.  For each  index $j\in \mathbb Z/q\mathbb Z$, let $P_j$ be the intersection point $\cap_{\ell=0}^{n-1} H_{j+\ell}$. By Definition \ref{weil},  the identity  of ideals in $\calO_S$
\begin{align}\label{ideal}
F_i(x_0,\hdots,x_n)\cdot (x_0,\hdots,x_n)=\prod_{j=i-n+1}^i (F_j(x_0,\hdots,x_n),\hdots,F_{j+n-1}(x_0,\hdots,x_n))
\end{align}
implies that for $v\notin S$ 
\begin{align}\label{divide}
\lambda_{ H_i,v}({\bf x})= \sum_{j=i-n+1}^i \lambda_{P_j,v}({\bf x}),
\end{align}
up to a $M_k$ constant. 
On the other hand,  let $\pi: X\to \mathbb P^n$ be the blow-up over the points $P_1,\hdots, P_q$ and let $\widetilde{H}_i\subset  X$ be the corresponding strict transform of $H_i$.
It follows from \eqref{Hi} that 
\begin{align}\label{equa}
\lambda_{ H_i,v}(\pi(Q))=\lambda_{\widetilde{H}_i,v}(Q)+\sum_{j=i-n+1}^i \lambda_{P_i,v}(\pi(Q)) 
\end{align}
up to a $M_k$ constant for  $Q\in X$, $v\not\in S$.  If $Q\notin \cup_{i=1}^q \supp (E_i)$, then $Q=\pi^{-1}({\bf x})$ for some ${\bf x}\ne P_i$, $1\le i\le q$.
Therefore, for  ${\bf x}:=[x_0:\cdots:x_n]\in\mathbb P^n(k)\setminus\{P_1,\hdots,P_q\}$ satisfying \eqref{ideal} we have
\begin{align}\label{integral}
\lambda_{\widetilde H_i,v}(\pi^{-1}({\bf x}))=0 \quad \text{up to a $M_k$ constant for   $v\not\in S$.}
\end{align}

By Theorem  \ref{proposition1}, there exists a Zariski closed subset $W$ of $X$ such that the set of points $\pi^{-1}({\bf x})$ satisfying \eqref{integral} are contained in $W$.  Therefore,   the points ${\bf x}\in\mathbb P^n(k)\setminus\{P_1,\hdots,P_q\}$  satisfying the identity \eqref{ideal} are contained in the Zariski closure of $\pi(W)$.
\end{proof}

\section{Degeneracy of holomorphic maps}\label{sec:analytic}
In this section, we give the analytic versions of the arithmetic statements obtained in the previous sections. This imply several results on Brody hyperbolicity.

\begin{theorem}
Let $n\geq 2$, $F_1,\dots, F_r, G \in \CC[X_1,\dots,X_n]$ be polynomials in general position (i.e. the associated hypersurfaces are in general position) with $\deg (F_i) \geq \deg (G)$ for $i=1,\dots,r$. Let $h_1,\dots,h_n$ be holomorphic functions on $\CC$ such that one of the following holds

\begin{itemize}
\item[{\rm (i)}]
$r\ge 2n$ and $\frac{G(h_1,\dots,h_n)}{F(h_1,\dots,h_n)}$ is holomorphic, for $ i=1,\hdots,r$; or
\item[{\rm (ii)}]
$r\ge n+1$ and $\frac{G(h_1,\dots,h_n)}{\prod_{i=1}^rF_i(h_1,\hdots,h_n)}$  is holomorphic.
\end{itemize}
Then $h_1,\dots,h_n$ are algebraically dependent.
\end{theorem}

This can be seen as a generalization of Borel's Theorem \cite{bor} stating that nowhere vanishing entire functions $h_1,\dots,h_{n+1}$ satisfying the identity $h_1+\dots+h_{n+1}=1$ are dependent. Indeed, we have the following corollary.

\begin{corollary}
Let $h_1,\dots,h_n$ be holomorphic functions on $\CC$ such that $\frac{1}{(h_1\dots h_n).(1-\sum_{i=1}^n h_i)}$ is holomorphic. Then $h_1,\dots,h_n$ are linearly dependent.
\end{corollary}

We recall the following definition.
\begin{defi}
We say that a complex variety $X$ is Brody pseudo-hyperbolic if there exists a proper closed subset $Z \subset X$ such that any (non-constant) entire curve $f:\CC \to X$ is contained in $Z$ i.e. $f(\CC) \subset Z.$
\end{defi}

Then we can rephrase in the analytic setting the main theorems of this paper.
 \begin{theorem}
Let $n\ge 2$, $r\ge 2n+1$ and  $D_0,D_1, \dots, D_{r} $ be  hypersurfaces  in general position on $\mathbb P^n(\CC)$.     
Let $\pi: X\to \mathbb P^n$ be the blowup long the union of subschemes $D_i\cap D_0$, $1\le i\le r$, and let $\widetilde D_i$ be the strict transform of  $D_i$.  Let $D=\widetilde D_1+\cdots+\widetilde D_r$.  Then $ X\setminus D$ is Brody pseudo-hyperbolic.
\end{theorem}

\begin{theorem}
Let $n\ge 2$ and  $H_1,\hdots,H_{2n}$ be $2n$ hyperplanes in general position on $\mathbb P^n(\CC)$.   Choose $n+1$ points $P_i$, $1\le i\le n+1$ such that $P_i\in H_i$, $1\le i\le n+1$, and $P_i\notin H_j$ if $i\ne j$ for $1\le j\le 2n$.  Let $ \pi: X\to\mathbb P^n$ be the blowup of the $n+1$ points $P_i$, $1\le i\le n+1$, and let $D\subset   X$ be the strict transform of $H_1+\cdots+H_{2n}$.   Then  
$  X\setminus D$ is Brody pseudo-hyperbolic.
\end{theorem}

 \begin{theorem}
Let  $n\ge 2$, $q\ge 3n$ be an integer; for every index $i\in \mathbb Z/q\mathbb Z$, let $H_i$ be a hyperplane in $\mathbb P^n(\CC)$.  Suppose that $H_i$'s are in general position.  Let for each  index $i\in \mathbb Z/q\mathbb Z$, $P_i$ be the intersection point $\cap_{j=0}^{n-1} H_{i+j}$.  Let $\pi: X\to \mathbb P^n$ be the blow-up over the points $P_1,\hdots, P_q$ and let $\widetilde H_i\subset  X$ be the corresponding strict transform of $H_i$ and let $D=\widetilde H_1+\cdots+\widetilde H_q$.  Then $ X\setminus D$ is Brody pseudo-hyperbolic.
\end{theorem}

The proofs of the above statements are the same as the arithmetic ones replacing Theorem \ref{Ru-Vojta} by its analytic analogue.
Its generalization is obtained using Vojta's version of Schmidt's subspace theorem \cite{Vojta89}, which gives a better control on the exceptional sets.
\begin{theorem}\label{SubspaceVojta2} Let $H_1,\hdots,H_q$ be hyperplanes in $\mathbb P^n(\CC)$ with the corresponding Weil functions $\lambda_{H_1},\hdots,\lambda_{H_q}.$ Then there exists a finite union of hyperplanes $Z$, depending only on $H_1,\hdots,H_q$, such that for any $\epsilon>0$, and any (non-constant) entire curve $f: \CC \to X$ with $f(\CC)\not \subset Z$
\begin{align}
\int_0^{2\pi} \max_{I}\sum_{i\in I}  \lambda_{H_i}(f(re^{i\theta})) \frac{d\theta}{2\pi} \leq_{\operatorname{exc}} (n+1+\e) T_f(r)
\end{align}
holds, where $\leq_{\operatorname{exc}}$ means that the inequality holds for all $r\in \mathbb{R}^+$ except a set of finite Lebesgue measure., where the maximum is taken over subsets $\{1,\hdots,q\}$ such that the linear forms defining $H_i$ for $i\in I$ are linearly independent.
\end{theorem}

By carefully tracing the exceptional sets with Theorem \ref{SubspaceVojta2}, the general analytic Theorem of Ru and Vojta can be stated as follows.

\begin{theorem}\cite[General Theorem (Analytic Part)]{ruvojta}\label{Ru-Vojta2}
    Let $X$ be a complex projective variety of dimension $n$ and let $D_1,\cdots,D_q$ be effective Cartier divisors intersecting properly on $X$. Let $\mathcal L$ be a big line bundle. Let $f: \CC \to X$ be a Zariski dense entire curve. Then, for every $\varepsilon>0$, there exists a  proper Zariski-closed subset $Z\subset X$, such that for any (non-constant) entire curve $f: \CC \to X$ with $f(\CC)\not \subset Z$,
    
    \begin{align*}
    \sum_{j=1}^q\beta_{\mathcal L, D_j}m_{f}(r,D_j) \leq_{\operatorname{exc}} (1+\varepsilon)T_{\mathcal L,f}(r)
\end{align*}
    holds,  where $\leq_{\operatorname{exc}}$ means that the inequality holds for all $r\in \mathbb{R}^+$ except a set of finite Lebesgue measure.
\end{theorem}

\section{Function Fields}\label{sec:ff}
In this section we give the analogue statements over function fields of the theorems obtained in the previous sections.
For this section we let $\kappa$ be an algebraically closed field of characteristic zero. Let $\calC$ be a non-singular projective curve defined over $\kappa$ and let $K = \kappa(\calC)$ denote its function field. We refer to \cite[Section 7.2]{RTW} for the basic definitions of heights and proximity functions in the function field setting. We recall the definition of algebraic hyperbolicity.

\begin{definition}\label{def:alg_hyp}
   Let $(X,D)$ be a pair of a non-singular projective variety $X$ defined over $\kappa$ and a normal crossing divisor $D$ on $X$. We say that $(X,D)$ is \emph{algebraically hyperbolic} if there exists an ample line bundle $\calL$ on $X$ and a positive constant $\alpha$ such that, for every non-singular projective curve $\calC$ and every morphism $\varphi: \calC \to X$ the following holds:
   \begin{equation}\label{eq:alg_hyp}
       \deg \varphi^* \calL \leq \alpha \cdot \left( 2g(\calC) - 2 + N^{[1]}_\varphi(D)\right),
   \end{equation}
   where $N^{[1]}_\varphi(D)$ is the cardinality of the support of $\varphi^*(D)$.

   We say that $(X,D)$ is \emph{pseudo algebraically hyperbolic} if there exists a proper closed subvariety $Z$ of $X$ such that \eqref{eq:alg_hyp} holds for every morphism $\varphi: \calC \to X$ such that $\varphi(\calC)$ is not contained in $Z$.
\end{definition}

We can now rephrase Theorems \ref{IP'}, \ref{proposition7} and \ref{Theorem7}.

\begin{theorem}\label{th:ff1.3}
Let $n\ge 2$, $r\ge 2n+1$ and  $D_0,D_1, \dots, D_{r} $ be  hypersurfaces  in general position on $\mathbb P^n$ defined over $\kappa$.
Let $\pi: X\to \mathbb P^n$ be the blowup long the union of subschemes $D_i\cap D_0$, $1\le i\le r$, and let $\widetilde D_i$ be the strict transform of  $D_i$.  Let $D=\widetilde D_1+\cdots+\widetilde D_r$.  Then $ X\setminus D$ is algebraically pseudo-hyperbolic.
\end{theorem}

\begin{theorem}\label{th:ff1.4}
Let $n\ge 2$ and  $H_1,\hdots,H_{2n}$ be $2n$ hyperplanes in general position on $\mathbb P^n$ defined over $\kappa$.   Choose $n+1$ points $P_i$, $1\le i\le n+1$ such that $P_i\in H_i$, $1\le i\le n+1$, and $P_i\notin H_j$ if $i\ne j$ for $1\le j\le 2n$.  Let $\pi: X\to\mathbb P^n$ be the blowup of the $n+1$ points $P_i$, $1\le i\le n+1$, and let $D\subset   X$ be the strict transform of $H_1+\cdots+H_{2n}$. Then $X\setminus D$ is  algebraically pseudo-hyperbolic.
\end{theorem}

\begin{theorem}\label{th:ff1.5}
    Let  $n\ge 2$, $q\ge 3n$ be an integer; for every index $i\in \mathbb Z/q\mathbb Z$, let $H_i$ be a hyperplane in $\mathbb P^n$ defined over $k$.  Suppose that $H_i$'s are in general position.  Let for each  index $i\in \mathbb Z/q\mathbb Z$, $P_i$ be the intersection point $\cap_{j=0}^{n-1} H_{i+j}$.  Let $\pi: X\to \mathbb P^n$ be the blow-up over the points $P_1,\hdots, P_q$ and let $\widetilde H_i\subset  X$ be the corresponding strict transform of $H_i$ and let $D=\widetilde H_1+\cdots+\widetilde H_q$.  Then $ X\setminus D$ is algebraically pseudo-hyperbolic.    
\end{theorem}

We remark that, even if we stated the results in the so-called \emph{split case}, our proofs carry over almost verbatim to the non-split case as well.

As in the analytic setting, the proofs of the above statements follow the same lines of the proof of our arithmetic results and the same strategy as in our previous paper \cite{RTW} with two modifications. On one hand we can use the results in Section \ref{sec:analytic} instead of \cite[Theorem 8.3 B]{levin_annal} for the case in which $2g(\calC) - 2 + N^{[1]}_\varphi(D)\leq 0$. On the other hand we replace the use of Theorem \ref{Ru-Vojta} with the following analogue that uses a version of the Schmidt subspace theorem over function fields recently obtained in \cite[Theorem 15]{GSW}. In particular this gives a better control on the exceptional set.

\begin{theorem} \label{thm:ffconst}
    Let $X\subset \PP^m$ be a projective variety over $\kappa$ of dimension $n$, let $D_1,\cdots,D_q$ be effective Cartier divisors intersecting properly on  $X$, and let $\calL$ be a big line sheaf.  
    Then for any $\epsilon>0$, there exist  constants $c_1$ and $c_2$, independent of the curve $\calC$ and the set $S$, and a finite collection of  hypersurfaces $\mathcal Z$ (over $\kappa$) in $\PP^m$ of degree at most $c_2$ such that for any map $x=[x_0:\cdots:x_m] :\calC\to X$, where $x_i\in K$, outside the augmented base locus of $\calL$ we have 
  either
  \begin{align*}
    \sum_{i=1}^q \beta_{\mathcal L, D_i} m_{D_i,S}(x)\le (1+\epsilon) h_{\calL}(x)+c_1 \max\left\{ 1, 2g(\calC)-2+|S|\right\},
 \end{align*}
 or  the image of $x$  is contained in  $\mathcal Z$ .
 
 \end{theorem}

 \begin{proof}   
The proof is similar to the first part of the proof of \cite[Theorem 7.6]{RTW}.  We will follow its argument and notation and only indicate the modification.
 Let $\epsilon>0$ be given.  Since $\calL$ is a big line sheaf, there is a constant $c$ such that  $\sum_{i=1}^qh_{D_i}(x)\le c h_{\calL}(x)$ for all $x\in X(K)$ outside the augmented base locus $B$ of $\calL$. 
 By the properties of the local heights, together with the fact that $m_{D_i,S}\le h_{D_i}+O(1)$, we can choose $\beta_i\in\Q$ for all $i$ such that 
 \[
 \sum_{i=1}^q (\beta_{\mathcal L, D_i}-\beta_i) m_{D_i,S}(x)\le \frac{\epsilon}{2}h_{\calL}(x)
 \]
 for all $x\in X\setminus B(K)$. Therefore, we can assume that $\beta_{\mathcal L, D_i}=\beta_i\in\Q$ for all $i$ and also that $\beta_i\ne 0$ for each $i$. From now on we will assume that the point $x \in X(K)$ does \emph{not} lie on $B$.

 Choose positive integers $N$ and $b$ such that
 \begin{equation}\label{aut_choices}
   \left( 1 + \frac nb \right) \max_{1\le i\le q}
       \frac{\beta_i Nh^0(X, \calL^N) }
         {\sum_{m\ge1} h^0(X, \calL^N(-mD_i))}
     < 1 + \epsilon\;.
 \end{equation}

 Then, using \cite[Theorem 7.5]{RTW} with the same notation, we obtain
  \begin{equation}\label{eq:weilbase_ff}
   \begin{split}
   &\frac b{b+n} \left( \min_{1\leq i \leq q} \sum_{m\ge 1}
     \frac{h^0(\calL^N(-mD_i))}{\beta_i} \right)
     \sum_{i=1}^q  \beta_i \lambda_{D_i,\mathfrak p} (x) \\
   &\qquad\le \max_{1\le i\le T_1} \lambda_{\mathcal B_i,\mathfrak p}  (x)+ O(1)= \max_{1\le i\le T_1} \sum_{j\in J_i} \lambda_{s_j,\mathfrak p} (x) + O(1).
   \end{split}
 \end{equation}
 Let $M=h^0(X,\calL^N)$, let the set $\{\phi_1,\dots,\phi_M\}$ be a basis of the vector space $H^0( X, \calL^N)$, and let 
 \begin{align}\label{Phi}
 \Phi=[\phi_1,\dots,\phi_M]:X \dashrightarrow \PP^{M-1}(\kappa)
 \end{align}
  be the corresponding rational map.  By  \cite[Theorem 15]{GSW},  there exists a finite collection of linear subspaces $\mathcal R$ over $\kappa$ such that, whenever  $\Phi\circ x$ is not in $\mathcal R$, we have the following
 \begin{equation}\label{eq:schmidt2_ff}
   \sum_{\mathfrak p\in S}\max_J \sum_{j\in J} \lambda_{s_j,\mathfrak p}(x) 
     \le M \, h_{\calL^N}(x)  +\frac{M(M-1)}2(2g-2+|S|),
 \end{equation}
 here the maximum is taken over all subsets $J$ of $\{1,\dots,T_2\}$
 for which the sections $s_j$, $j\in J$, are linearly independent (with the same notation as in the proof of \cite[Theorem 7.1]{RTW}).
 We first consider when $\phi_1,\dots,\phi_M$ are linearly  independent over $\kappa$.
 Combining \eqref{eq:weilbase_ff} and \eqref{eq:schmidt2_ff} gives
 $$\sum_{i=1}^q   \beta_i m_{D_i,S}(x)
   \leq \left( 1 + \frac nb \right) \max_{1\le i\le q}
     \frac{\beta_i }
       {\sum_{m\ge1} h^0(\calL^N(-mD_i))} \, M\, h_{\calL^N}(x)  +c_1'(2g-2+|S|) + O(1),$$
       where $c_1'=\frac{M(M-1)}2. $
  Using
 \eqref{aut_choices} and the fact that $h_{\calL^N}(x)=Nh_{\calL}(x)$,
 we have
 \begin{equation*}
  \sum_{i=1}^q   \beta_i m_{D_i,S }(x)
  \leq \left( 1+\epsilon \right)h_{\calL}(x)+c_1'(2g-2+|S|) + O(1),
  \end{equation*}
  which implies the first case of the Theorem.
 
  To conclude we note that, if  $\Phi\circ x$ is  in one of the linear subspace of $\mathcal R$ over $\kappa$  in $\PP^{M-1}$, then $a_1\phi_1(x) + \dots + a_M \phi_M(x) = 0$, where $H=\{a_1z_1 + \dots + a_M z_M = 0\}$ is one of the hyperplanes (over $\kappa$) in $\PP^{M-1}$ coming from $\mathcal R$.
  
 On the other hand, since $\phi_1,\dots,\phi_M$ is a basis of $H^0(X,\calL^N)$, it follows that $\Phi(X)$ is not contained in $H$, hence   $x(\calC)$ is contained in  is the hypersurface  coming from $a_1\phi_1 + \dots + a_M \phi_M = 0$  in $\PP^{m}$ (as $X\subset \mathbb P^m$) whose degree is bounded independently of $\calC$ and $x$ as wanted.   Moreover, since $\mathcal R$ is a finite collection of  linear subspaces over $\kappa$  in $\PP^{M-1}$, the number of $H$ is finite and hence the number of hypersurfaces obtained above is also finite.
 
 \end{proof}

\nocite{}

\bibliography{references}{}
\bibliographystyle{alpha}
\end{document}